%% file: amr2.tex
\documentclass[a4paper,12pt]{article}
\usepackage{amsmath,amssymb} \usepackage{euler,eucal}
\usepackage{ifxetex}
\ifxetex
  \usepackage{fontspec} \usepackage{polyglossia}
  \defaultfontfeatures{Mapping=tex-text}
  \setromanfont[Alternate=1]{URW Palladio L}
\else
  \usepackage[utf8]{inputenc}
  \usepackage[T1]{fontenc}
  \usepackage[english]{babel}
  \usepackage{tgpagella}
\fi
\usepackage[unicode=true,plainpages=false]{hyperref}
\usepackage{graphicx}
\hypersetup{colorlinks=false,pdfborder={0 0 0.1},linkcolor=magenta,anchorcolor=magenta,urlcolor=blue,citecolor=blue}

\usepackage[a4paper]{geometry}
\usepackage{algorithmic}
\usepackage[ruled]{algorithm}
\usepackage{amsthm}
\theoremstyle{definition}

\newtheorem{remark}{Remark}

\newtheorem{lemma}{Lemma}

\newtheorem{theorem}{Theorem}

\usepackage{tikz} \usetikzlibrary{positioning,shapes}
\usepackage{pgfplots} \usetikzlibrary{plotmarks}

\let\eps=\varepsilon

\let\le=\leqslant
\let\ge=\geqslant
\let\leq=\leqslant
\let\geq=\geqslant

\def\i{\mathbf{i}}  \def\e{\mathbf{e}}

\def\P{\mathcal{P}}
\def\S{\mathcal{S}} \def\X{\mathcal{X}}  \def\Z{\mathcal{Z}}
\def\O{\mathcal{O}}

\def\vec{\mathop{\mathrm{vec}}\nolimits}
\def\als{\mathop{\mathrm{ALS}}\nolimits}

\def\const{\mathop{\mathrm{const}}\nolimits}
\def\Span{\mathop{\mathrm{span}}\nolimits}
\def\trans{*}
\def\tru{\star}
\def\lmax{\lambda_{\mathrm{max}}}
\def\lmin{\lambda_{\mathrm{min}}}

\def\R{\mathbb{R}} 

\def\step{h}
\begin{document}
\author{Sergey V. Dolgov%
\thanks{Max-Planck-Institut f{\"u}r Mathematik in den Naturwissenschaften, Inselstr. 22-26, D-04103 Leipzig, Germany, 
and Institute of Numerical Mathematics, Russian Academy of Sciences, Gubkina str. 8, 119333 Moscow, Russia ({\tt sergey.v.dolgov@gmail.com})} 
~and Dmitry V. Savostyanov%
\thanks{University of Southampton, Department of Chemistry, Highfield Campus, Southampton SO17 1BJ, United Kingdom ({\tt dmitry.savostyanov@gmail.com})}}
\title{Alternating minimal energy methods for linear systems in higher dimensions. Part II: Faster algorithm and application to nonsymmetric systems%
\thanks{Partially supported by
         RFBR grants  12-01-00546-a, 11-01-12137-ofi-m-2011, 11-01-00549-a, 12-01-33013, 12-01-31056,
         Russian Fed. Gov. contracts No. $\Pi 1112$, 14.740.11.0345, 16.740.12.0727
         and EPSRC grant EP/H003789/1  at the University of Southampton.
This work was initiated when D.S. was with the Institute of Numerical Mathematics RAS, Moscow.         
}}
\date{April 11, 2013}
\maketitle

\begin{abstract}
In this paper we accomplish the development of the fast rank--adaptive solver for tensor--structured symmetric positive definite linear systems in higher dimensions.
In~\cite{ds-amr1-2013} this problem is approached by alternating minimization of the energy function, which we combine with steps of the basis expansion in accordance with the steepest descent algorithm.
In this paper we combine the same steps in such a way that the resulted algorithm works with one or two neighboring cores at a time.
The recurrent interpretation of the algorithm allows to prove the global convergence and to estimate the convergence rate.
We also propose several strategies, both rigorous and heuristic, to compute new subspaces for the basis enrichment in a more efficient way.
We test the algorithm on a number of high--dimensional problems, including the non-symmetrical Fokker--Planck and chemical master equations, for which the efficiency of the method is not fully supported by the theory.
In all examples we observe a convincing fast convergence and high efficiency of the proposed method.

\end{abstract}


\section{Introduction}
In this paper we develop the results of~\cite{ds-amr1-2013}.
We consider tensor--structured linear systems, which arise naturally from high--dimensional problems, e.g. PDEs.
The number of unknowns grows exponentially w.r.t. the number of dimensions $d,$ which makes standard algorithms inefficient even for moderate $d.$
This problem is known as the \emph{curse of dimensionality}, and is attacked by different low--parametric approximations, e.g. \emph{sparse grids}~\cite{smolyak-1963,griebel-sparsegrids-2004} and \emph{tensor product methods}~\cite{kolda-review-2009,khor-survey-2011,hackbusch-2012}.
A particularly simple, elegant and efficient representation of high--dimensional data is a linear tensor network, also called the~\emph{matrix product states} (MPS) and~\emph{tensor train} (TT) format.

The MPS approach was originally proposed in the quantum physics community to represent the quantum states  of many--body systems~\cite{fannes-mps-1992,klumper-mps-1993}.
This representation was re-discovered as the TT format by Oseledets and Tyrtyshnikov~\cite{osel-tt-2011}, who were looking for a proper method to generalize a low--rank decomposition of matrices to high--dimensional arrays (tensors).
The MPS approach came with the \emph{alternating least squares} (ALS) and \emph{density matrix renormalization group} (DMRG)~\cite{white-dmrg-1993,ostlund-dmrg-1995} algorithms for the ground state problem.
The ALS considers the minimization of the Rayleigh quotient over the vectors with a fixed tensor structure, while DMRG does the same allowing the rank of the solution to change.
Experiments from quantum physics point out that the convergence of the DMRG is usually notably fast, while the one of the ALS can be rather poor.

The general numerical linear algebra context in which the TT format is introduced allows to think more widely about the power of tensor representations.
For instance, we can apply DMRG--like techniques to high--dimensional problems other than just the ground state problem, e.g. interpolation of high-dimensional data~\cite{ot-ttcross-2010,so-dmrgi-2011proc}, solution of linear systems~\cite{holtz-ALS-DMRG-2012,DoOs-dmrg-solve-2011}, fast linear algebra in tensor formats~\cite{Os-mvk2-2011}.
We can also consider better alternatives to the DMRG, which follow the same \emph{alternating linear scheme} (ALS) framework, but are numerically more efficient.
A tempting goal is to obtain an algorithm which has the DMRG-like convergence and the ALS-like numerical complexity.
In~\cite{ds-amr1-2013} we present such an algorithm for a solution of symmetric positive definite (SPD) linear systems in higher dimensions.

The central idea in~\cite{ds-amr1-2013} is to support the alternating steps, i.e. optimization in a fixed tensor manifold, by steps which \emph{expand} the basis in accordance with some classical iterative algorithms.
A \emph{steepest descent} (SD) algorithm is a natural choice for SPD problems.
The \emph{enrichment} step uses the essential information about the \emph{global} residual of the large high--dimensional system on the local optimization step, that helps to escape the spurious local minima introduced by the nonlinear tensor formulation and ensure the global convergence.
The convergence rate of the whole method can then be established adapting a classical theory.
In contrast, optimization in the fixed tensor manifolds can be analyzed via the Gauss--Seidel theory and only \emph{local} convergence estimates are available~\cite{ushmaev-tt-2013}, which hold only in a (very) small vicinity of the exact soution.

The global enrichment step used in algorithms  ``$t+\als(z)$'' and ``$\als(t+z)$'' in~\cite{ds-amr1-2013} modifies all components of the tensor train format simultaneously.  
There is nothing particularly wrong with this, but it is interesting to mix the same steps differently to obtain the algorithm which works with only one or two neighboring components at once, similarly to the DMRG technique.
In this paper we develop such a method, namely the \emph{alternating minimal energy} (AMEn) algorithm.
We prove the global convergence of AMEn and estimate the convergence rate w.r.t. the one of the steepest descent algorithm.
We also propose several methods to compute the required local component of the global residual, using either the SVD--based approximation, or incomplete Cholesky decomposition, or low--rank ALS approximation.

The rest of the paper is organized as follows.
In Section \ref{sec:def} we introduce necessary definitions and notations.
In Section \ref{sec:amen} we propose the AMEn algorithm, then we compare it with similar algorithms from~\cite{ds-amr1-2013} and prove the convergence theorem.
In Section \ref{sec:prac} we discuss efficient methods to compute the required component of the residual.
In Section~\ref{sec:num} we test the algorithm on a number of high--dimensional problems, including the non-symmetrical Fokker--Planck and chemical master equations, for which the efficiency of the method is not fully supported by the theory.
In all examples we observe a convincing fast convergence and high efficiency of the proposed method, as well as the advantages of the AMEn algorithm over the previously proposed ones.

\section{Notations and definitions} \label{sec:def}
This paper is based on the notations of~\cite{ds-amr1-2013}, which we recall briefly here.

We consider linear systems $Ax=y$ in $d$--dimensional space, i.e. assume that a vector $x$ has $d$ indices $i_1,\ldots,i_d,$ and $i_k=1,\ldots,n_k,$ $k=1,\ldots,d.$
Such arrays are referred to as $d$--tensors $\mathbf{x}=[\mathbf{x}(i_1,\ldots,i_d)].$
For the purposes of this paper it is convenient to consider a \emph{vectorization} of a tensor
\begin{equation}\nonumber
 x = \vec\mathbf{x}, \qquad x(\overline{i_1\ldots i_d})=\mathbf{x}(i_1,\ldots,x_d),
\end{equation}
where $\overline{i_1\ldots i_d}$ denotes a single index combined from $i_1,\ldots,i_d$%
\footnote{
The multi--index can be defined following the \emph{big--endian} convention
$\overline{i_1\ldots i_d}=i_d+(i_{d-1}-1)n_d +\ldots+(i_1-1)n_2\ldots n_d$
or \emph{little--endian} convention
$\overline{i_1\ldots i_d}=i_1+(i_2-1)n_1 +\ldots+(i_d-1)n_1\ldots n_{d-1}.$
The big--endian notation is similar to numbers written in the positional system, while the {little--endian} notation is used in numerals in the Arabic scripts and is consistent with the \textsc{Fortran} style of indexing.
The definition of the Kronecker (tensor) product $\otimes$ should be also consistent with the chosen endianness.
The orthodox definition in linear algebra assumes the big--endianness, while the development of the efficient program code usually makes us think in the little--endian way.
The rest of the paper can be read without a particular care of the endianness.
It is enough to remember that $z = x\otimes y$ means $z(\overline{ij})=x(i)y(j).$
}.
This \emph{index grouping} is widely used throughout the paper.
In the following we do not distinguish between $\mathbf{x}$~and~$x.$

The tensor train (TT) representation of $x$ is written as the following multilinear map%
\footnote{
Note that $\tau$ maps tensor train cores to a vectorized representation of a $d$--tensor, not to the $d$--tensor itself, cf. \emph{quantized} tensor train (QTT)~\cite{khor-qtt-2011}.
In this paper we do not distinguish between them and keep the notation simple.
},
\begin{equation}\label{eq:tt}
 \begin{split}
  x = \tau(\bar X) & = \tau(X^{(1)},\ldots,X^{(d)}), \\
  x(\overline{i_1\ldots i_d}) & = X^{(1)}_{\alpha_0,\alpha_1}(i_1) X^{(2)}_{\alpha_1,\alpha_2}(i_2) \ldots X^{(d-1)}_{\alpha_{d-2},\alpha_{d-1}}(i_{d-1}) X^{(d)}_{\alpha_{d-1},\alpha_d}(i_d),
 \end{split}
\end{equation}
where $i_k$ are referred to as \emph{mode} (physical) indices, $\alpha_k=1,\ldots,r_k$ are the \emph{rank} indices, $X^{(k)}$ are the tensor train \emph{cores} (TT--cores) and $\bar X = (X^{(1)}, \ldots, X^{(d)})$ denotes the tensor train.
We follow the Einstein summation convention, which assumes a summation over every pair of repeated indices.
All equations are supposed to hold for all possible values of free (unpared) indices.

The $\tau$ mapping is defined also for a subset of TT--cores (a \emph{subtrain}) and maps it to the \emph{interface} matrix of size $n_1\ldots n_k \times r_k,$ defined as follows,
\begin{equation}\label{eq:iface}
 \begin{split}
  X^{\leq k} & = \tau(X^{(1)},\ldots,X^{(k)}), \\
  X^{\leq k}(\overline{i_1 i_2 \ldots i_k}, \alpha_k) & = X^{(1)}_{\alpha_1}(i_1) X^{(2)}_{\alpha_1\alpha_2}(i_2) \ldots X^{(k)}_{\alpha_{k-1},\alpha_k}(i_k),
 \end{split}
\end{equation}
and similarly for symbols  $X^{<k}, X^{>k}$ and  $X^{\geq k}.$
For $x$ given by~\eqref{eq:tt} we have $X^{\{k\}}=X^{\leq k} X^{> k}.$

Note that the definition of $\tau$ allows us to write
$$
x = \tau(X^{\leq k},X^{>k}) = \tau(X^{< k},X^{(k)},X^{> k}) = \tau(X^{<k},X^{(k)},X^{(k+1)},X^{\geq k+2}),
$$
where two last mappings depict the decompositions used in ALS and DMRG algorithms proposed by S. White et al.~\cite{white-dmrg-1993,ostlund-dmrg-1995} for the ground state problem in Quantum Physics.
The original DMRG algorithm is formulated via the minimization of the Rayleigh quotient, where the heavily nonlinear high--dimensional optimization is reduced to the sequence of numerically tractable optimizations over the elements of each core.

Similarly, we consider the solution of a linear equation $Ax=y$ through the minimization of the \emph{energy} function
\begin{equation}\label{eq:energy}
 J(x) = \|x_\tru - x \|_A^2 =  (x,Ax) - 2\Re(x,y) + \const,
\end{equation}
where $x_\tru=A^{-1}y$ is the exact solution, and $\|u\|_A^2 = (u,u)_A=(u,Au)$ denotes the $A$--norm of a vector $u.$
Following the \emph{alternating linear scheme} (ALS), the high-dimensional minimization is reduced to the minimization w.r.t. all cores one-by-one.
Each \emph{local} minimization is equivalent to the solution of a linear system, which is tractable due to a moderate size.
The high-dimensional linear system can be split into a sequence of one-dimensional systems due to the \emph{linearity} of the tensor train format $\tau(X^{(1)},\ldots,X^{(d)})$ w.r.t. each TT--core $X^{(k)}.$
This linearity writes as the following matrix-by-vector product
\begin{equation}\label{eq:lin}
 \begin{split}
   x = \tau(\bar X)   & = \tau(X^{< k},X^{(k)},X^{> k}) = \P_{\neq k}(\bar X) x^{(k)}, \\
 \P_{\neq k}(\bar X)  & = X^{<k} \otimes I_{n_k} \otimes \left(X^{>k} \right)^\top,
 \end{split}
\end{equation}
where the vectorized TT--core $x^{(k)}$ is a reshape of the three-dimensional array into a vector
\begin{equation}\label{eq:vec}
 x^{(k)} = \vec X^{(k)}, \qquad x^{(k)}(\overline{\alpha_{k-1}i_k\alpha_k}) = X^{(k)}_{\alpha_{k-1},\alpha_k}(i_k).
\end{equation}
The elementwise definition of the \emph{frame matrix} $\X_{\neq k} = \P_{\neq k}(\bar X)$ is the following
\begin{equation}\label{eq:proj}
 \X_{\neq k}(\overline{i_1\ldots i_d}, \overline{\alpha_{k-1} j_k \alpha_k}) =  X^{(1)}_{\alpha_1}(i_1) \ldots X^{(k-1)}_{\alpha_{k-2},\alpha_{k-1}}(i_{k-1}) \delta(i_k,j_k)X^{(k+1)}_{\alpha_{k},\alpha_{k+1}}(i_{k+1}) \ldots X^{(d)}_{\alpha_{d-1}}(i_d),
\end{equation}
where $\delta(i,j)$ is the Kronecker symbol, i.e., $\delta(i,j)=1$ if $i=j$ and $\delta(i,j)=0$ elsewhere.
Similarly we define frame matrices $\P_{\leq k}(\bar X),$ $\P_{\geq k}(\bar X),$ $\P_k(\bar X).$
For example, $\P_{\leq k}(\bar X)$ writes
\begin{equation}\label{eq:pk}
 \begin{split}
  \X_{\leq k} = \P_{\leq k}(\bar X)  & = X^{\leq k} \otimes I_{n_{k+1} \ldots {n_d}}, \\
  \X_{\leq k}(\overline{i_1\ldots i_d}, \overline{\alpha_k j_{k+1}\ldots j_d }) & =  X^{(1)}_{\alpha_1}(i_1) \ldots X^{(k)}_{\alpha_{k-1},\alpha_k}(i_k) \delta(i_{k+1},j_{k+1}) \ldots \delta(i_d,j_d).
 \end{split}
\end{equation}

The $A$--orthogonal projector on the subspace $\Span U$ is defined as follows
\begin{equation}\label{eq:aort}
  R_U =  U (U^\trans A U)^{-1} U^\trans A.
\end{equation}
It is easy to check that $R_U^2 = R_U$ and $R_U U = U.$
Also, for any $v$ such that $(U,v)_A=0$ it holds $R_Uv=0,$ hence the name $A$--orthogonal.

\section{Alternating minimal energy methods} \label{sec:amen}
\subsection{AMEn and ALS}
One of the main results of the previous paper~\cite{ds-amr1-2013} is the algorithm $\als(t+z).$
Each iteration of this algorithm consists of one `global' basis enrichment step which changes all the cores, followed by $d$ update steps over all the cores subsequently.
Classical optimization algorithms for tensor networks, e.g. ALS and DMRG, follow the alternating linear framework, i.e. update one or two neighboring cores at a time.
We would like to keep the enrichment as well as other steps within the same idea, and
propose another version of the method as Alg.~\ref{alg:amen}, which we will refer to as the \emph{alternating minimal energy algorithm} (AMEn).
The difference between two algorithms is illustrated by a simple three--dimensional example in Fig.~\ref{fig:alsamen}.

To analyse the convergence of the $\als(t+z)$ algorithm, we can see it as a method which implements the (approximate) steepest descent step followed by a sequence of optimization steps for the energy function.
The approximate steepest descent step $x=t+\step\tilde z$ with $\tilde z\approx z=y-At$ and optimal $\step$ gives~\cite[Thm.~1]{ds-amr1-2013} the convergence rate $\omega_{\tilde z} = \omega_z + \O(\eps),$ where $\omega_z$ denotes the progress of the exact steepest descent step.
A fortiori, the \emph{global convergence} of the $\als(t+z)$ is proven with the convergence rate not slower than $\omega_{\tilde z}.$
The AMEn algorithm does not have a global enrichment step and the convergence can not be proven in one line.
However, the convergence analysis is possible if Alg.~\ref{alg:amen} is seen as a recurrent method.
Though the theoretical estimates do not provide a clear distinction which method is preferable, in numerical experiments in Sec.~\ref{sec:num} we will observe that the AMEn technique delivers more accurate solution, while the average convergence rate is almost the same as of the $\als(t+z)$ method.

\begin{figure}[h]
 \begin{center}
  \resizebox{.95\textwidth}{!}{\input{./tikz/alsamen.tikz}}
 \end{center}
 \caption{Schematic representation of ALS(t+z) and AMEn algorithms} \label{fig:alsamen}
\end{figure}
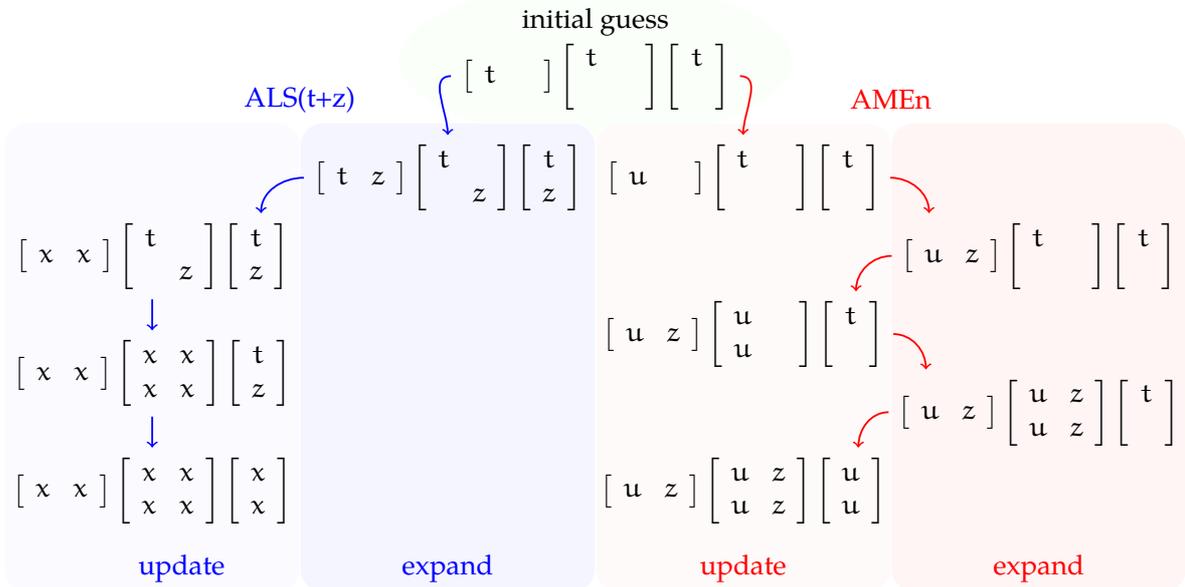

\subsection{AMEn in two dimensions}
\begin{figure}[t]
 \begin{center}
  \resizebox{.95\textwidth}{!}{\input{./tikz/amen2.tikz}}
 \end{center}
 \caption{Schematic representation of the AMEn algorithm}
 \label{fig:amen}
\end{figure}
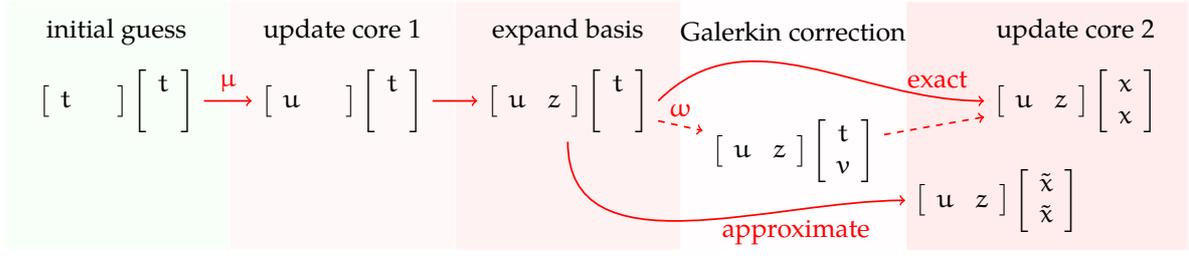

The idea of the convergence analysis is introduced by a two--dimensional example, see Fig.~\ref{fig:amen}.
In two dimensions Alg.~\ref{alg:amen} can be seen as a sequence of the following operations.
\begin{enumerate}
 \item Start from the initial guess
   \begin{equation}\nonumber
      t =
        \begin{bmatrix} T^{(1)} &  \phantom{Z^{(1)}} \end{bmatrix}
        \begin{bmatrix} T^{(2)} \\ \phantom{Z^{(2)}} \end{bmatrix}
   \end{equation}
 \item Update the first TT--core, minimizing the energy function over the entries of $T^{(1)}$
   \begin{equation}\nonumber
      U^{(1)} = \arg\min_{T^{(1)}} J\left(
        \begin{bmatrix} T^{(1)} &  \phantom{Z^{(1)}} \end{bmatrix}
        \begin{bmatrix} T^{(2)} \\ \phantom{Z^{(2)}} \end{bmatrix}
        \right), \qquad
      u =
        \begin{bmatrix} U^{(1)} &  \phantom{Z^{(1)}} \end{bmatrix}
        \begin{bmatrix} T^{(2)} \\ \phantom{Z^{(2)}} \end{bmatrix}
   \end{equation}
 \item Expand the basis in the first core using the first TT--core of the residual.
   \begin{equation}\nonumber
      y-Au = z \approx \tilde z=\tau(Z^{(1)},Z^{(2)}), \qquad
      u =
        \begin{bmatrix} U^{(1)} &  Z^{(1)} \end{bmatrix}
        \begin{bmatrix} T^{(2)} \\ \phantom{Z^{(2)}} \end{bmatrix}
   \end{equation}
 \item Perform the Galerkin correction step by minimizing the energy function over the bottom part of second TT--core.
   \begin{equation}\nonumber
      V^{(2)} = \arg\min_{V} J\left(
        \begin{bmatrix} U^{(1)} &  Z^{(1)} \end{bmatrix}
        \begin{bmatrix} T^{(2)} \\ V \end{bmatrix}
        \right), \qquad
      v =
        \begin{bmatrix} U^{(1)} &  Z^{(1)} \end{bmatrix}
        \begin{bmatrix} T^{(2)} \\ V^{(2)} \end{bmatrix}
   \end{equation}
 \item Minimize the energy function over all entries of the second TT--core
   \begin{equation}\nonumber
      X^{(2)} = \arg\min_{X} J\left(
        \begin{bmatrix} U^{(1)} &  Z^{(1)} \end{bmatrix}
        X
        \right), \qquad
      x =
        \begin{bmatrix} U^{(1)} &  Z^{(1)} \end{bmatrix}
        X^{(2)}
   \end{equation}
\end{enumerate}

The ALS update steps 2 and 5 reduce the energy function by a factor $\mu_1^2\leq1$ and $\mu_2^2\leq1,$ respectively, which can be rigorously estimated only locally, i.e. in a very small vicinity of a true solution.
The basis enrichment step 3 does not provide any progress, because it does not change the solution vector, but only its TT--representation.
The Galerkin correction step 4 does not technically present in the algorithm.
If we omit it, the update step 5 will deliver the same TT--core, optimizing the energy function over the positions occupied by both upper and bottom parts of the second core.
Without actually affecting the result of the computations, step 4 is essential to analyse the convergence of the whole method, since the progress of the Galerkin correction step can be estimated w.r.t. the one of the steepest descent.
This idea is formally expressed as the following theorem.

\begin{theorem}\label{thm:amen2}
 In the notations set above for the two-dimensional linear system $Ax=y$ one iteration of AMEn Alg.~\ref{alg:amen} provides the following progress
 \begin{equation}\nonumber
   \frac{J(x)}{J(t)}
       = \frac{\|x_\tru-x\|_A^2}{\|x_\tru-t\|_A^2}
       = \frac{\|x_\tru-u\|_A^2}{\|x_\tru-t\|_A^2}
       \:\frac{\|x_\tru-x\|_A^2}{\|x_\tru-u\|_A^2}
       = \mu^2 \omega_\X^2,
 \end{equation}
$$
\mu\leq1, \qquad \omega_\X \leq \omega_\Z \leq \omega_{\tilde z} = \omega_z + \O(\eps),
$$
 where $\X = \P_1(\bar X) = X^{(1)} \times I_{n_2},$ $\Z = \P_1(\bar Z) = Z^{(1)} \times I_{n_2},$ and $\omega_\S$ is defined by
 \begin{equation}\label{eq:omegas}
   \omega_\S^2 = 1 - \frac{(x_\tru-u,R_\S(x_\tru-u))_A}{(x_\tru-u,x_\tru-u)_A}.
 \end{equation}
\end{theorem}
\begin{proof}
The minimization in step 5 is written as $x^{(2)}_\tru = \arg\min_{x^{(2)}} J(\X x^{(2)}),$ $x^{(2)} = \vec X^{(2)},$ cf.~\eqref{eq:lin},
and the gradient is zero when the Galerkin conditions $(\X^\trans A \X) x^{(2)}_\tru = \X^\trans y$ are met.
The solution after one AMEn iteration writes as follows
\begin{equation}\label{eq:Xx}
 x=\X x^{(2)}_\tru = \X (\X^\trans A \X)^{-1} \X^\trans A x_\tru = R_\X x_\tru,
\end{equation}
where $R_\X$ is the $A$--orthogonal projector on $\X,$ cf.~\eqref{eq:aort}.
Since $u \in \Span\X$ we have
\begin{equation}\label{eq:astep}
 x_\tru-x = \left( I - R_\X \right) (x_\tru-u), \qquad \frac{J(x)}{J(u)}=\frac{\|x_\tru-x\|_A^2}{\|x_\tru-u\|_A^2}= 1 - \frac{(x_\tru-u,R_\X(x_\tru-u))_A}{(x_\tru-u,x_\tru-u)_A}=\omega_\X^2.
\end{equation}
Similarly, the progress of the Galerkin correction step 4 is estimated as follows (see~\cite{ds-amr1-2013}),
\begin{equation}\label{eq:gstep}
 x_\tru-v = \left( I - R_\Z \right) (x_\tru-u), \qquad \frac{J(v)}{J(u)}=\frac{\|x_\tru-v\|_A^2}{\|x_\tru-u\|_A^2}= 1 - \frac{(x_\tru-u,R_\Z(x_\tru-u))_A}{(x_\tru-u,x_\tru-u)_A}=\omega_\Z^2.
\end{equation}
It is not easy to estimate $\omega_\X$ directly.
However, since $Z^{(1)}$ is a part of $X^{(1)},$ we have $\Span\Z\subset\Span\X$ and therefore $\omega_\X\leq\omega_\Z.$
Similarly, since $\tilde z \in\Span\Z$ it holds $\omega_\Z\leq\omega_{\tilde z},$ where $\omega_{\tilde z}$ denotes a progress of the perturbed steepest descent step.
The final estimate for $\omega_{\tilde z}$ is obtained in~\cite[Thm 1]{ds-amr1-2013} with a precise derivation of the asymptotic $\O(\eps)$ term.
To finish the proof we note that $J(u)/J(t) = \mu^2 \leq 1$ by construction of step 2.
\end{proof}
\begin{remark}\label{rem:less1}
 A convergence rate $\omega_z$ of the steepest descent algorithm is estimated using the Kantorovich inequality as follows
 \begin{equation}\nonumber
  \omega_z \leq \frac{\lmax(A)-\lmin(A)}{\lmax(A)+\lmin(A)} = \Omega(A) < 1,
 \end{equation}
 where $\lmin(A)$ and $\lmax(A)$ denote the smallest and largest eigenvalues of $A.$
 For any $\omega_z$ we can choose such a threshold level $\eps$ that $\omega_{\tilde z}<1,$ which guarantees the global convergence of the AMEn algorithm.
\end{remark}


\begin{algorithm}[t] 
 \caption{AMEn algorithm} \label{alg:amen}
 \begin{algorithmic}[1]
  \REQUIRE System $Ax=y$ in the TT--format, initial guess $t=\tau(\bar T)$
  \ENSURE Updated solution $x = \tau(\bar X)$
  \FOR[Cycle over TT--cores]{$k=1,\ldots,d$}
    \STATE Update $U^{(k)} = \arg\min_{S^{(k)}} J(\tau(X^{(1)},\ldots,X^{(k-1)},S^{(k)},T^{(k+1)},\ldots,T^{(d)}))$,
    \STATE For local problem~\eqref{eq:reduced} approximate $z_k=y_k-A_k u_k \approx \tilde z_k=\tau(Z_k^{(k)},\ldots,Z_k^{(d)})$
    \STATE Expand the basis $X^{(k)}=\begin{bmatrix}U^{(k)} & Z^{(k)}_k\end{bmatrix}, \quad S^{(k+1)} := \begin{bmatrix}T^{(k+1)} \\ 0\end{bmatrix}$
    \STATE Recover the TT--orthogonality of $\bar X_k = (X^{(1)},\ldots,X^{(k)},S^{(k+1)},T^{(k+2)},\ldots,T^{(d)})$.
  \ENDFOR
  \RETURN $x = \tau(X^{(1)},\ldots,X^{(d)})$
 \end{algorithmic}
\end{algorithm}

\subsection{AMEn in higher dimensions}
In higher dimensions, the AMEn Alg.~\ref{alg:amen} can be described using the same scheme.
We start from an initial guess $t=\tau(T^{(1)},T^{\geq2}).$
In step 2 we update the first core obtaining $u=\tau(U^{(1)},T^{\geq2}).$
In step 3 we approximate the residual $y-Au=z\approx\tilde z=\tau(Z^{(1)},Z^{\geq2})$ and expand the first core by $Z^{(1)}.$
These operations are numerically tractable, see Sec.~\ref{sec:prac} for details.
Minimization problem in step 5 leads to the following linear system
\begin{equation}\label{eq:loc1}
 (\X_1^\trans A \X_1) x^{\geq 2}  = \X_1^\trans y, \qquad \X_1=\P_1(\bar X)=X^{(1)}\otimes I_{n_2\ldots n_d}, \qquad x^{\geq 2} = \tau(X^{(2)},\ldots,X^{(d)}),
\end{equation}
which has $r_1n_2\ldots n_d$ unknowns and is still too large to be solved directly.
Note that $\X_1 = X^{(1)} \otimes I$ is a rank-1 multilevel matrix, hence $B=\X_1^\trans A \X_1$ has the same TT--ranks as $A$, but its TT--representation is shorter by one core.
Similarly, $\X_1^\trans y$ represents a smaller right-hand side in the TT format, and the solution $x^{\geq2}$ is sought in the TT format as well.
Therefore, the linear problem in $d$ dimensions is reduced to the one in $d-1$ dimensions, i.e. to the minimization over the remaining subtrain, and the same algorithm is applied recurrently.

The convergence rate of AMEn is defined by a recurrent application of the result of Thm.~\ref{thm:amen2}.
We need to consider a sequence of the reduced problems
\begin{equation}\label{eq:reduced}
 A_k x^{\geq k} = y_k, \qquad A_k = \X_{<k}^\trans A \X_{<k}, \quad y_k=\X_{<k}^\trans y, \qquad
 \X_{<k}=\P_{<k}(\bar X).
\end{equation}
The initial guess is $t^{\geq k}=\tau(T^{(k)},T^{\geq k+1}),$
the update step for the core $T^{(k)}$ gives $u_k=\tau(U^{(k)},T^{\geq k+1}),$
further steps of AMEn return the solution $x^{\geq k} = \tau(X^{(k)},\ldots,X^{(d)}),$
and the true solution is defined by $x^{\geq k}_\tru = A_k^{-1} y_k.$
Similarly to Thm.~\ref{thm:amen2} we have
\begin{equation}\label{eq:omegak}
  \frac{\|x^{\geq k}_\tru - u_k\|_{A_k}^2}{\|x^{\geq k}_\tru - t^{\geq k}\|_{A_k}^2}  = \mu_k^2 \leq 1, \qquad
  \frac{\|x^{\geq k}_\tru - x^{\geq k}\|_{A_k}^2}{\|x^{\geq k}_\tru - u_k\|_{A_k}^2}  = \omega_k^2 = 1 - \frac{(c_k,R^{\langle A_k \rangle}_\X c_k)_{A_k}}{\|c_k\|_{A_k}^2},
\end{equation}
where $c_k=x^{\geq k}_\tru - u_k,$
$\X = \P_1(X^{(k)},\ldots,X^{(d)}),$
and $R^{\langle A_k \rangle}_\X$ denotes the $A_k$--orthogonal projector to $\X.$
With these definitions in hand we write the following theorem.

\begin{theorem}\label{thm:amen}
The AMEn Alg.~\ref{alg:amen} converges globally with the following convergence rate
 \begin{equation}\label{eq:rate}
  \begin{split}
   \frac{\|x_\tru-x\|_A^2}{\|x_\tru-t\|_A^2} & =
                 \mu_1^2 \biggl(  \omega_1^2 + (1-\omega_1^2)
                 \mu_2^2 \Bigl(  \omega_2^2 + (1-\omega_2^2)
                 \mu_3^2 \bigl( \omega_3^2 + \ldots
               + \mu_{d-1}^2 \omega_{d-1}^2 \bigr) \ldots \Bigr)\biggr)
     \\ &    = \sum_{k=1}^{d-1} \omega_k^2 \prod_{j=1}^{k-1}(1-\omega_j^2) \prod_{j=1}^k \mu_j^2
     = \phi_d^2.
  \end{split}
 \end{equation}
\end{theorem}
\begin{proof}
In two dimensions, the theorem reduces to Thm.~\ref{thm:amen2}, which proves the base of recursion.
We suppose that~\eqref{eq:rate} holds in $d-1$ dimension and prove it recurrently for $d$ dimensions.

If $t=\tau(T^{(1)},T^{\geq2})$ is the initial guess, $u=\tau(U^{(1)},T^{\geq2})$ appears after the update of the first core, and $x=\tau(X^{(1)},X^{\geq2})$ is the result returned by one iteration of AMEn, the error $x_\tru-x$ is written as follows
\begin{equation}\nonumber
   x_\tru - x  = x_\tru - \tau(X^{(1)}, X^{\geq2}_\tru) + \tau(X^{(1)}, X^{\geq2}_\tru) - \tau(X^{(1)}, X^{\geq2})
               = \left( x_\tru - \X_1 x^{\geq2}_\tru \right) + \X_1 \left( x^{\geq2}_\tru - x^{\geq2} \right),
\end{equation}
where the first term is the error of the first (`outer') AMEn step provided the solution $x^{\geq2}_\tru$ of the reduced problem~\eqref{eq:loc1} is computed exactly,
and the second term is the error of other (`inner') AMEn steps, which we find using the assumption of the recurrence.
Using~\eqref{eq:astep} we show that these terms are $A$--orthogonal,
\begin{equation}\label{eq:errdec}
 \begin{split}
   x_\tru - x          & = \left( I - R_{\X_1} \right) x_\tru + \X_1 \left( x^{\geq2}_\tru - x^{\geq2} \right), \\
   \|x_\tru - x\|_A^2  & = \left\| (I - R_{\X_1}) x_\tru\right\|_A^2 + \left\|\X_1 (x^{\geq2}_\tru - x^{\geq2}) \right\|_A^2.
 \end{split}
\end{equation}

The first term writes by Thm.~\ref{thm:amen2} as follows,
\begin{equation}\nonumber
 \frac{\| (I-R_{\X_1}) x_\tru \|_A^2}{\|x_\tru - t\|_A^2} =
 \frac{\| (I-R_{\X_1}) (x_\tru-u) \|_A^2}{\|x_\tru - t\|_A^2} =
 \frac{\| x_\tru - u \|_A^2}{\|x_\tru - t\|_A^2} \: \frac{\| (I-R_{\X_1}) (x_\tru-u) \|_A^2}{\|x_\tru - u\|_A^2}
 = \mu_1^2 \omega_1^2.
\end{equation}

For the second term we need the following norm equivalence,
$$
 \left\|\X_1(x^{\geq2}_\tru-x^{\geq2})\right\|_A^2 =
 \left(x^{\geq2}_\tru-x^{\geq2}, \X_1^\trans A \X_1 (x^{\geq2}_\tru-x^{\geq2})  \right) =
 \left\|x^{\geq2}_\tru-x^{\geq2}\right\|_B^2,
$$
where $B=\X_1^\trans A \X_1.$
In the right--hand side we see the error norm of the AMEn algorithm applied to the linear problem~\eqref{eq:loc1} with $d-1$ cores.
According to our assumption, it writes by~\eqref{eq:rate} as follows,
\begin{equation}\label{eq:rate1}
   \frac{\|x^{\geq2}_\tru-x^{\geq2}\|_B^2}{\|x^{\geq2}_\tru-t^{\geq2}\|_B^2}
       = \sum_{k=2}^{d-1} \omega_k^2 \prod_{j=1}^{k-1}(1-\omega_j^2) \prod_{j=1}^k \mu_j^2
       =\phi_{d-1}^2,
\end{equation}
and for the second term we obtain
\begin{equation}\nonumber
  \frac{\|\X_1(x^{\geq2}_\tru-x^{\geq2})\|_A^2}{\|x_\tru-t\|_A^2} =
   \phi_{d-1}^2 \frac{\|x^{\geq2}_\tru-t^{\geq2}\|_B^2}{\|x_\tru-t\|_A^2}
   = \phi_{d-1}^2 \mu_1^2
   \frac{\|\X_1(x^{\geq2}_\tru - t^{\geq2})\|_A^2}{\|x_\tru-u\|_A^2}.
\end{equation}
In the numerator we simplify $\X_1 x^{\geq2}_\tru = R_{\X_1} x_\tru,$ cf.~\eqref{eq:Xx}, and
\begin{equation}\nonumber
  \X_1 t^{\geq2}  = \tau(X^{(1)},t^{\geq2})
                        = \tau\left(
                        \begin{bmatrix} U^{(1)} & Z^{(1)} \end{bmatrix}
                        \begin{bmatrix} t^{\geq2} \\ \phantom{z^{\geq2}} \end{bmatrix}
                        \right)
                        = \tau(U^{(1)},t^{\geq2})
                        = u = R_{\X_1}u.
\end{equation}
Finally, we write last part of the seconf term as follows, cf.~\eqref{eq:astep},
\begin{equation}\nonumber
   \frac{\|\X_1(x^{\geq2}_\tru-t^{\geq2})\|_A^2}{\|x_\tru-u\|_A^2} =
   \frac{\|R_{\X_1}(x_\tru-u)\|_A^2}{\|x_\tru-u\|_A^2} = 1-\omega_1^2.
\end{equation}

Substituting both terms into~\eqref{eq:errdec}, we complete the proof by
\begin{equation}\label{eq:rec}
   \frac{\|x_\tru - x\|_A^2}{\|x_\tru-t\|_A^2} =
   \frac{\| (I - R_{\X_1}) x_\tru\|_A^2}{\|x_\tru-t\|_A^2} +
   \frac{\|\X_1 (x^{\geq2}_\tru - x^{\geq2})\|_A^2}{\|x_\tru-t\|_A^2} =
   \mu_1^2\omega_1^2 + \mu_1^2 (1-\omega_1^2) \phi_{d-1}^2 = \phi_d^2,
\end{equation}
where $\phi_{d-1}$ evaluates the convergence rate of the AMEn in $d-1$ dimensions by~\eqref{eq:rate1}.
\end{proof}
\begin{remark}
 From the recurrence relation~\eqref{eq:rec} it is clear that if the convergence of the reduced problem $\phi_{d-1}<1,$ then the convergence rate of the full problem
 $$
 \phi_d^2=\mu_1^2\omega_1^2 + \mu_1^2 (1-\omega_1^2) \phi_{d-1}^2<1.
 $$
 By Remark~\ref{rem:less1}, we can always choose the approximation threshold $\eps$ to ensure $\phi_2<1$ for the AMEn algorithm in two dimensions.
 Therefore, we can always choose $\eps$ to provide $\phi_d<1,$ which guarantees the global convergence of Alg.~\ref{alg:amen}.
\end{remark}

 Similarly to Thm.~\ref{thm:amen2} we can estimate $\omega_k$ in~\eqref{eq:omegak} as follows,
\begin{equation}\nonumber
  \omega_k^2
       = 1 - \frac{(c_k,R^{\langle A_k \rangle}_\X c_k)_{A_k}}{\|c_k\|_{A_k}^2}
    \leq 1 - \frac{(c_k,R^{\langle A_k \rangle}_\Z c_k)_{A_k}}{\|c_k\|_{A_k}^2}
    \leq 1 - \frac{(c_k,R^{\langle A_k \rangle}_{\tilde z_k} c_k)_{A_k}}{\|c_k\|_{A_k}^2} = \omega_{\tilde z_k}^2,
\end{equation}
where $\tilde z_k\approx z_k=A_kc_k=y_k-A_ku_k,$  $\tilde z_k=\tau(Z^{(k)},Z^{\geq k+1})$ and $\Z=\P_1(Z^{(k)},Z^{\geq k+1}).$
In the right--hand side we see the convergence rate of the perturbed steepest descent method applied to the reduced problem with the matrix $A_k.$
It is estimated~\cite[Thm 1]{ds-amr1-2013} as follows
\begin{equation}\nonumber
 \omega_{\tilde z_k} = \omega_{\tilde z_k} + \O(\eps), \qquad
 \omega_{z_k} \leq \frac{\lmax(A_k)-\lmin(A_k)}{\lmax(A_k)+\lmin(A_k)} = \Omega(A_k) < 1,
\end{equation}
where $\eps$ denotes the relative accuracy of $\tilde z_k.$
It can be shown that if all $\X_{<k}$ are orthogonal we have
$$
\lmin(A_k) \geq \lmin(A_{k-1}), \qquad \lmax(A_k) \leq \lmax(A_{k-1}),
$$
\begin{equation}\label{eq:Omegak}
\Omega(A_{d-1}) \leq \Omega(A_{d-2}) \leq \ldots \leq \Omega(A_1) = \Omega(A) = \frac{\lmax(A)-\lmin(A)}{\lmax(A)+\lmin(A)},
\end{equation}
where the last term estimates the convergence rate of SD algorithm applied to $Ax=y.$

\begin{remark}
The requirement for $\X_{<k}$ to be orthogonal for $k=1,\ldots,d$ is equivalent to the \emph{TT--orthogonality} of the tensor train $\bar X,$ see~\cite{osel-tt-2011} for details.
As a counterpart of \eqref{eq:Omegak}, we may say that this requirement prevents the condition numbers of reduced matrices in \eqref{eq:reduced} from increasing, which is essential for numerical stability.
By construction, $X^{(k)} = \begin{bmatrix} U^{(k)} & Z^{(k)}\end{bmatrix}$ does not provide the TT--orthogonality of $\bar X.$
An additional step is required to \emph{recover} the orthogonality, i.e. make the TT--core $X^{(k)}$ column--orthogonal.
It is done via a QR decomposition
\begin{equation}\nonumber
   \begin{bmatrix}U^{(k)} & Z_k^{(k)}\end{bmatrix} \begin{bmatrix}T^{(k+1)} \\ 0\end{bmatrix}
 = \begin{bmatrix}\hat U^{(k)} & \hat Z^{(k)}\end{bmatrix}
   \begin{bmatrix}R_{uu} & R_{uz} \\ & R_{zz}\end{bmatrix}
   \begin{bmatrix}T^{(k+1)} \\ 0\end{bmatrix}
 = \begin{bmatrix}\hat U^{(k)} & \hat Z^{(k)}\end{bmatrix}
   \begin{bmatrix}R_{uu} T^{(k+1)} \\ 0\end{bmatrix}.
\end{equation}
We denote the result after orthogonalization by the same symbol
 $X^{(k)}=\hat X^{(k)}=\begin{bmatrix}\hat U^{(k)} & \hat Z^{(k)}\end{bmatrix}.$
All considerations in Theorem \ref{thm:amen} remain valid, since all estimates are based on the subspaces, which are unaffected by the QR decomposition.
Therefore, we imply the TT--orthogonality silently to simplify the discussion, and the actual operation is fast and does not influence the analysis.
\end{remark}

In \cite[Thm. 3]{ds-amr1-2013}, the convergence rate  of the \emph{greedy descent} algorithm $t+\als(z)$ is given by exactly the same formula as~\eqref{eq:rate}, but the values $\omega_k$ are defined differently.
In the AMEn method, $\omega_k$ is given by~\eqref{eq:omegak} and relates to the convergence of the reduced problem~\eqref{eq:reduced}.
For the algorithm $t+\als(z)$ it is defined as $\omega_k=\omega_{\Z_{\leq k}}$ in terms of~\eqref{eq:omegas}.
Since $\Span\Z_{\leq k+1}\in\Span\Z_{\leq k}$, for the greedy algorithm it holds $\omega_{k+1}\leq\omega_k.$
For the AMEn algorithm, we can prove this only for upper bounds $\Omega(A_k)$ as shown in~\eqref{eq:Omegak}.

Considering the `width' of $\Z_{\leq k}$ and $\X_{\leq k}$, we can expect that
$$
\omega_k(\mathrm{AMEn}) \ll \omega_k(t+\als(z)), \qquad
\mu_k(\mathrm{AMEn}) \leq \mu_k(t+\als(z)),
$$
which is observed in numerical experiments.
It is not clear however whether this heuristic statement holds in general.

\section{Fast approximation of the residual} \label{sec:prac}
In this section we discuss how to compute the approximation on step 3 in AMEn Alg.~\ref{alg:amen} efficiently.

\subsection{SVD--based approximation}
In steepest descent schemes proposed in \cite{ds-amr1-2013}, the low-rank approximation of the residual is computed once per iteration, and a standard SVD-based TT-rounding procedure from~\cite{osel-tt-2011} can be used.
In AMEn Alg.~\ref{alg:amen} we can not approximate each $z_k$ individually by the TT-SVD, since it makes the total complexity quadratic in the dimension $d$.
To keep the complexity linear in $d,$ we have to investigate the tensor structure of $z_k.$

Looking at the TT representation of the reduced system \eqref{eq:reduced} and recalling that $\X_{<k} = X^{<k} \otimes I_{n_k \cdots n_d}$ has a rank-one structure, we conclude that $y_k = \X_{<k}^\trans y$ inherits the blocks $k+1,\ldots,d$ from $y$ as follows,
$$
y_k(\overline{\alpha_{k-1}i_k,\ldots,i_d}) = Y^{(k)}_k (\overline{\alpha_{k-1}i_k}) Y^{(k+1)}(i_{k+1}) \ldots Y^{(d)}(i_d),
$$
where $Y^{(k)}_k = \P^\trans_{<k}(X^{(1)},\ldots,X^{(k)}) Y^{\leq k}$.
The similar representation holds for the local matrix $A_k = \tau(A^{(k)}_k,A^{(k+1)}, \ldots, A^{(d)})$.
Therefore the local residual $z_k=y_k-A_ku_k$ writes
\begin{equation}
z_k(\overline{\alpha_{k-1}\,i_k\ldots i_d}) = \hat Z_k^{(k)} (\overline{\alpha_{k-1}i_k}) \cdot \hat Z^{(k+1)}(i_{k+1}) \cdots \hat Z^{(d)}(i_d),
\label{eq:res_red}
\end{equation}
$$
\hat Z_k^{(k)} (\overline{\alpha_{k-1} i_k})  = \begin{bmatrix}Y^{(k)}_k(\overline{\alpha_{k-1}i_k}) & -A^{(k)}_k(\overline{\alpha_{k-1}i_k}, \overline{\beta_{k-1}j_k}) \otimes U^{(k)}(\overline{\beta_{k-1}j_k})\end{bmatrix},
$$
\begin{equation}\nonumber 
\hat Z^{(p)}(i_p)  =  \begin{bmatrix}Y^{(p)}(i_p) \\ & A^{(p)}(i_p,j_p)\otimes T^{(p)}(j_p)\end{bmatrix},
 \qquad p=2,\ldots,d-1,
\end{equation}
\begin{equation}\nonumber 
\hat Z^{(d)}(i_d)  =  \begin{bmatrix}Y^{(d)}(i_d) \\ A^{(d)}(i_d,j_d)\otimes T^{(d)}(j_d)\end{bmatrix}.
\end{equation}
The TT decomposition of the exact residual $z_k=\tau(\hat Z_k^{(k)},\hat Z^{(k+1)},\ldots,\hat Z^{(d)})$ has only one block which actually depends on the information gained in the step $k,$ the others can be precomputed before the iteration.
In the approximate residual $\tilde z_k=\tau(Z_k^{(k)},\ldots,Z_k^{(d)})$ all blocks depend on the recently computed $U^{(k)},$ but only one block $Z_k^{(k)}$ is actually required.
This means that if we keep all TT--cores $Z^{(p)}$ right--orthogonal, we can compute $Z^{(k)}_k$ by the SVD compression of $\hat Z^{(k)}_k$ only.
Therefore, the SVD--based approximation of the residual involves the information from only one core, and the complexity of each enrichment step does not grow with $d.$
The overall complexity is therefore linear in $d,$ as required.



\subsection{Cholesky--based  approximation}
The singular value decomposition provides the optimal approximation accuracy for a prescribed rank, but is numerically expensive.
Each TT--core $\hat Z^{(k)}$ has the sizes $R_{k-1} \times n_k \times R_k$, where $R_k = r_k r_k(A) + r_k(y),$ and the QR and SVD operations have the complexity $\O(R^3) = \O(r^6),$  which may be inefficient.
Since a very precise approximation of the residual is not always required, we may avoid expensive QR and SVD steps by considering the Unfinished Cholesky algorithm (see, e.g.~\cite{sav-rr-2009})  applied to the Gram matrix of the first unfolding of $z_k$.
A careful implementation allows to reduce the complexity to $\O(r^5)$.

Given \eqref{eq:res_red}, its first unfolding reads
$$
Z^{\{k\}}_k(\overline{\alpha_{k-1}i_k}, \overline{i_{k+1},\ldots,i_d}) = \hat Z_k^{(k)} (\overline{\alpha_{k-1}i_k}) \cdot \hat Z^{(k+1)}(i_{k+1}) \cdots \hat Z^{(d)}(i_d),
$$
and the Gram matrix $G_k = Z^{\{k\}}_k (Z^{\{k\}}_k)^\trans$ computes as follows,
$$
G_k(\overline{\alpha_{k-1}i_k}, \overline{\beta_{k-1}j_k}) = \Gamma_k^{(k)}(\overline{\alpha_{k-1}i_k}, \overline{\beta_{k-1}j_k})\cdot \Gamma^{(k+1)} \cdots \Gamma^{(d)},
$$
where
$
\Gamma_k^{(k)}(\overline{\alpha_{k-1}i_k}, \overline{\beta_{k-1}j_k}) = \hat Z_k^{(k)}(\overline{\alpha_{k-1}i_k}) \otimes \hat Z_k^{(k)}(\overline{\beta_{k-1}j_k}),
$
and
$
\Gamma^{(p)} = \hat Z^{(p)}(i_{p}) \otimes \hat Z^{(p)}(i_{p})
$
for $p=k+1,\ldots,d$.

Similarly to the SVD--based method, we can precompute $\Gamma^{(p)}$ before the iteration.
The enrichment vectors are then calculated as the factors $Z_k^{(k)}$  in the Unfinished Cholesky decomposition
$$
G_k(\overline{\alpha_{k-1}i_k}, \overline{\beta_{k-1}j_k}) \approx Z_k^{(k)}(\overline{\alpha_{k-1}i_k}) D \left(Z_k^{(k)}(\overline{\beta_{k-1}j_k})\right)^\trans.
$$

\subsection{ALS--based approximation}
To reduce the complexity even further, we can approximate $\tilde z \approx z=y-Au$ using the auxiliary ALS iteration.
We start from some low-rank initial guess $\tilde z = \tau(\bar Z)$ and minimize $\|\tilde z - z\|$ under the constraint $\tilde z = \Z_{\neq k} z^{(k)},$ where $z^{(k)} = \vec Z^{(k)}.$
For a unitary $\Z_{\neq k}$ this leads to the extremal condition $z^{(k)} = \Z_{\neq k}^\trans z.$

Until the convergence of the fixed-rank ALS is not proved, this approach is heuristic.
However, as was observed in numerical experiments, it provides the enrichment basis almost of the same quality as the SVD-based method, although  much faster.
It is enough to conduct two alternating methods simultaneously step by step, which means that only one ALS update is performed for $\tilde z$ between the subsequent AMEn iterations.

The algorithm is organized as follows.
Given some low-rank approximation $\tilde z$, we assume that  $\mathrm{span}(Z^{>k})$ is a good approximation basis for $z_k$ as well (which appears to hold in practice).
That is, the enrichment is computed as a projection
\begin{equation} \label{eq:rich_als}
\vec Z_k^{(k)} = \P_{\neq 1}^\trans(Z^{(k)},\ldots,Z^{(d)})z_k.
\end{equation}
Now, we need to update $\tilde z$ for the forthcoming iterations.
The current solution approximant is $u = \tau(X^{<k},U^{(k)},T^{>k})$, so the TT--core $Z^{(k)}$ writes
\begin{equation} \label{eq:res_als}
 \begin{split}
 z^{(k)} & = \Z_{\neq k}^\trans (y-Au) = \Z_{\neq k}^\trans y - \Z_{\neq k}^\trans A \P_{\neq k}(\bar U) u^{(k)},
 \\
 \bar U & =(X^{(1)},\ldots,X^{(k-1)},U^{(k)},T^{(k+1)},\ldots,T^{(d)}).
 \end{split}
\end{equation}
Note that $Z^{(k)}$ serves only as an update of the global residual approximation $\tilde z$ and cannot be used as an enrichment directly.

Similarly to the previous sections, one may avoid $\O(d^2)$ cost of \eqref{eq:res_als},\eqref{eq:rich_als} by performing all calculations involving the same TT blocks (e.g. $\hat Z^{(p)}$) only once during the AMEn iteration.
The resulting complexity is therefore that of the fixed-rank ALS, $\mathcal{O}(\rho n^2 r^3)$, where $\rho = r(\tilde z)$ is the TT-rank of $\tilde z$.
In practice, it is usually enough to take $\rho \le 5$.

Finally, let us note that the minimization $\|\tilde z - z\|$ is equivalent to the maximization of $(\tilde z, z)$. In other words, the AMEn method solves approximately the following \emph{minimax} problem,
$$
\max\limits_{\tilde z} \min\limits_{x} (\tilde z, y-Ax) = \max\limits_{\tilde z} \min\limits_{x} (\tilde z, x_\tru-x)_A,
$$
by performing the subsequent ALS updates for $\tilde z$ in the rank-$\rho$ TT format, and $x$ in the TT format with a \emph{varying} rank $r$.
This allows us to establish a connection between the AMEn method and the greedy approximations, in particular, the Minimax Proper Generalized Decomposition \cite{nouy-modred-2010}.
However, the greedy techniques usually perform the optimization over rank-1 separable tensors.
As a some improvement one may mention the \emph{orthogonal} greedy method, which orthogonalizes the residual to the basis of $R$ current canonical factors of the solution (i.e. selects $R$ scalars).
The AMEn approach may be considered as a next milestone in the family of adaptive tensor-structured linear solvers.
By updating a larger portion of solution data at a time, it appears to be more robust and accurate, as was demonstrated in~\cite{ds-amr1-2013}.


\section{Numerical experiments} \label{sec:num} 

In these experiments, we compare the MATLAB versions of AMEn algorithms proposed above (SVD, Chol, ALS) with the $\als(t+z)$ method from the previous work \cite{ds-amr1-2013}, as well as the DMRG method from \cite{DoOs-dmrg-solve-2011}.
The AMEn and DMRG methods were implemented within the framework of the TT-Toolbox\footnote{\url{http://github.com/oseledets/TT-Toolbox}}~2.2 (routines \texttt{amen\_solve2} and \texttt{dmrg\_solve3}, respectively),
and the computations were done at the Linux machine with 2.6 GHz AMD Opteron CPU, and MATLAB R2012a.

\subsection{SPD example: Poisson equation}
First, we consider the same symmetric positive definite example as in \cite{ds-amr1-2013}.
This is the high--dimensional Poisson equation,
\begin{equation}\nonumber
-\Delta x = e, \qquad x \in \Omega = [0, 1]^d, \qquad \left. x\right|_{\partial\Omega}=0,
\end{equation}
where $\Delta$ is the finite difference Laplacian discretization on a uniform grid with $64$ points in each direction, i.e., the total size of the system is $64^d$, and $e$ is the vector of all ones.
Different greedy-type methods were compared in \cite{ds-amr1-2013}, as well as the $\als(t+z)$ method. Now we focus on non-greedy techniques, including AMEn+SVD and AMEn+ALS, see Fig. \ref{fig:lap_amen}.
The TT--rank of the enrichment $\tilde z$ is $\rho=4$, Frobenius-norm threshold for the solution $\mathtt{tol} = 10^{-5}$, and the problem dimension $d=16$.

We see that all methods except the DMRG demonstrate comparable performances.
Even though in the beginning of the iterations the $\als(t+z)$ method seems to be the fastest,
it approaches the same CPU times as the AMEn methods when the rank increases.
In this example, the desired accuracy level is reached by all algorithms.
However, it might be not the case, as we will see in the following.

\begin{figure}[t]
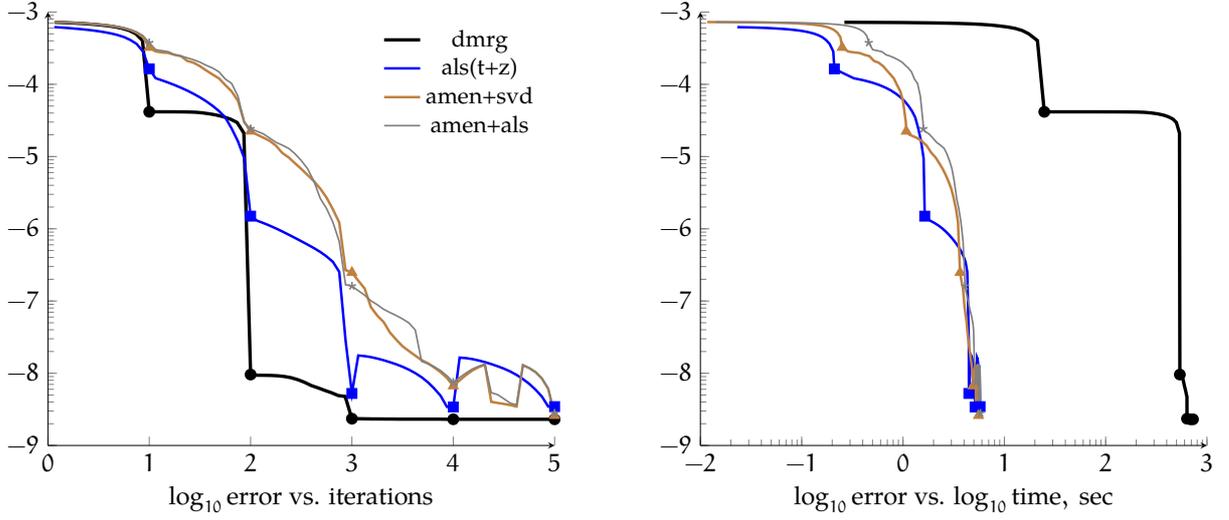

 \begin{center}  \def\nnn{16}
  \resizebox{.48\textwidth}{!}{\input{./Pic/pgfart.sty} \input{./Pic/conv_lp_i_2.tikz}}\hfill
  \resizebox{.48\textwidth}{!}{\input{./Pic/pgfart.sty} \input{./Pic/conv_lp_t_2.tikz}}
 \end{center}
\caption{$A$--norm error in different methods versus iterations (left), and CPU time (right). Poisson equation, $d=16$}
\label{fig:lap_amen}
\end{figure}

\subsection{Nonsymmetric example: Chemical Master Equation}
The second example is the
Chemical Master Equation \cite{vankampen-stochastic-1981}, applied to the $d$-dimensional cascade gene regulatory model \cite{hegland-cme-2007,Ammar-cme-2011}.
This is the huge-sized ODE
$$
\frac{d\psi(t)}{d t} = A \psi(t) \in \R^{n^d},
$$
where $\psi(t)=\{\psi(\i,t)\}$, $\i = (i_1,\ldots,i_d) \in [0,\ldots,63]^{\otimes d}$, so that $n=64$,
and the operator is formulated as follows,
$$
\begin{array}{rcl}
 A & = & A_1 + \cdots + A_{d}, \\
 A_1 \psi(\i,t) & = & \alpha_0 \cdot (\psi(\i-\e_1,t)-\psi(\i,t)) + \delta \cdot ((i_1+1) \psi(\i+\e_1,t) - i_1 \psi(\i,t)), \\
 A_k \psi(\i,t) & = & \dfrac{\beta i_{k-1}}{\beta i_{k-1} + \gamma} (\psi(\i-\e_k,t)-\psi(\i,t)) + \delta \cdot ((i_k+1) \psi(\i+\e_k,t) - i_k \psi(\i,t))
\end{array}
$$
for $k=2,\ldots,d,$ where $\e_k \in \R^d$ is the $k$-th identity vector.
The particular model parameters were fixed to the values
$$
\alpha_0=0.7, \quad \delta=0.07, \quad \beta = 1, \quad \gamma=5.
$$

\begin{figure}[t]
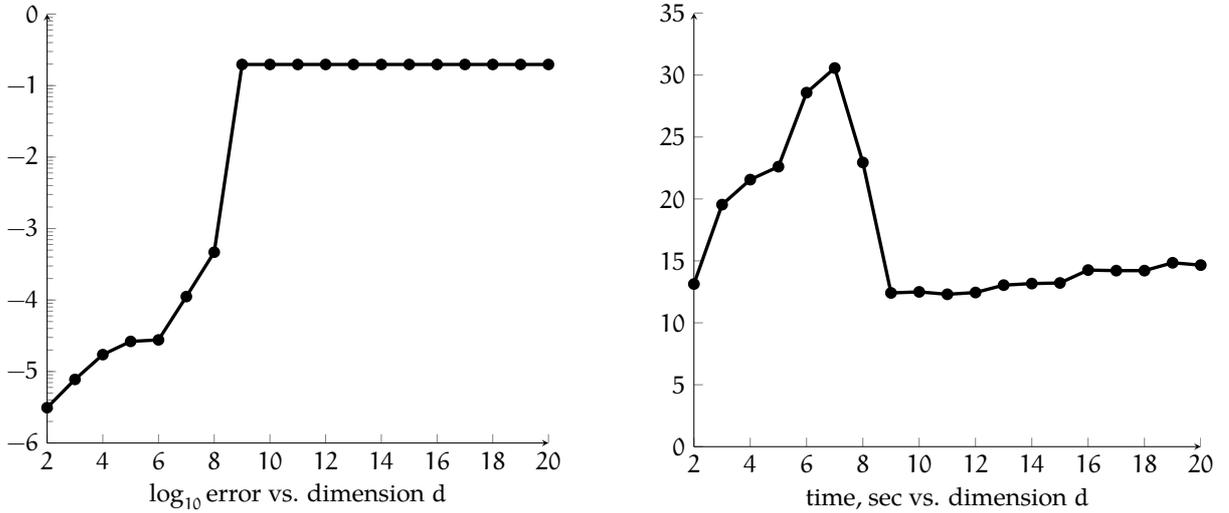

 \begin{center}
  \resizebox{.48\textwidth}{!}{\input{./Pic/pgfart.sty} \input{./Pic/resids_cme20_dmrg.tikz}}\hfill
  \resizebox{.48\textwidth}{!}{\input{./Pic/pgfart.sty} \input{./Pic/ttimes_cme20_dmrg.tikz}}
 \end{center}
\caption{Error (left) and CPU time (right) of the DMRG method vs. the dimension}
\label{fig:casc_dmrg}
\end{figure}

\begin{figure}[t]
 \begin{center}
  \resizebox{.48\textwidth}{!}{\input{./Pic/pgfart.sty} \input{./Pic/resids_cme20_symm_i.tikz}}\hfill
  \resizebox{.48\textwidth}{!}{\input{./Pic/pgfart.sty} \input{./Pic/resids_cme20_symm_t.tikz}}
 \end{center}
\begin{center} \hfil
\resizebox{.48\textwidth}{!}{\input{./Pic/pgfart.sty} \input{./Pic/errors_cme20_symm_i.tikz}}\hfill
\resizebox{.48\textwidth}{!}{\input{./Pic/pgfart.sty} \input{./Pic/errors_cme20_symm_t.tikz}}
\end{center}
\caption{Residual (top) and error (bottom) in different methods vs. iterations (left), and CPU time (right), 20-dimensional problem}
\label{fig:casc_amen_symm}
\end{figure}

\begin{figure}[t]
 \begin{center}
  \resizebox{.48\textwidth}{!}{\input{./Pic/pgfart.sty} \input{./Pic/errors_cme20_amen_i.tikz}}\hfill
  \resizebox{.48\textwidth}{!}{\input{./Pic/pgfart.sty} \input{./Pic/errors_cme20_amen_t.tikz}}
 \end{center}
\caption{Error in different methods vs. iterations (left), and CPU time (right), 20-dimensional problem}
\label{fig:casc_var_amen}
\end{figure}
\begin{figure}[t]
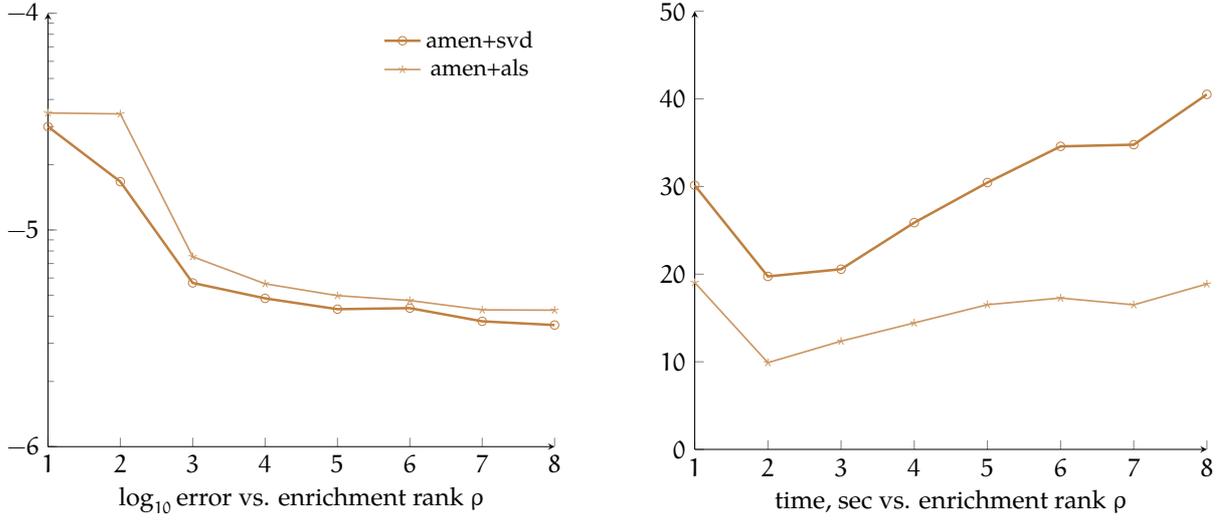

 \begin{center}
  \resizebox{.48\textwidth}{!}{\input{./Pic/pgfart.sty} \input{./Pic/errors_cme20_rho.tikz}}\hfill
  \resizebox{.48\textwidth}{!}{\input{./Pic/pgfart.sty} \input{./Pic/ttimes_cme20_rho.tikz}}
 \end{center}
\caption{Error (left) and CPU time (right) in AMEn methods vs. enrichment rank, 20-dimensional problem}
\label{fig:casc_amr_rho}
\end{figure}

The Chemical Master Equation serves as an accurate model for gene transcription, protein production and other biological processes.
However, its straightforward solution becomes impossible rapidly with increasing number of species $d$.
Existing techniques include the Monte-Carlo-type methods (so-called SSA \cite{gillespie-ssa-1976} and its descendants), as well as more tensor-related ones: Sparse Grids \cite{hegland-cme-2007}, greedy approximations in the canonical tensor format \cite{Ammar-cme-2011} and tensor manifold dynamics \cite{jahnke-cme-2008}.
The first two approaches only relax the curse of dimensionality to some extent; typical examples involve up to 10 dimensions and may take from 15 minutes to many hours on high-performance machines.
Tensor-product low-rank approaches seem to be more promising.
Unfortunately, we cannot estimate a possible potential of greedy or manifold dynamics methods, whereas up to now our alternating linear solution technique appears to be more efficient.
For more intensive study of the CME applications of the AMEn and DMRG methods see \cite{dkh-cme-2012} and \cite{kkns-cme-2013}, respectively.
Note that for systems with moderate dimensions and smaller time steps, the DMRG method can be of a good use for such problems, as was demonstrated in \cite{kkns-cme-2013}.
However, as we will see, the AMEn algorithm appears to perform better than DMRG for more complicated problems.

Two specific tricks allow to take more benefits from the tensor structuring.
First, we employ the Crank-Nicolson discretization in time, but instead of the step-by-step propagation, consider the time as a $(d+1)$-th variable and formulate one global system encapsulating all time steps \cite{DKhOs-parabolic1-2012},
$$
\left(I - \frac{\tau}{2} A\right) \psi(t_{m+1}) = \left(I + \frac{\tau}{2} A\right) \psi(t_{m}), \quad m=0,\ldots,N_t-1,
$$
where $t_m = \tau m$, $\tau$ is the time step size, and the initial state is $\psi(0) = \bigotimes_{k=1}^d e_1$, $e_1 \in \R^{64}$ is the first identity vector.
In particularly, we choose $N_t=2^{12}$, $T=10$, and $\tau=T/N_t$.
Such a time interval is not enough to reach the stationary solution, but the transient process is also of interest.
As a result, we end up with a $(d+1)$-dimensional system of size $n^d \cdot N_t$.

Second, we prepare all the initial data and seek the solution not in the $(d+1)$-dimensional TT-format directly, but in the so-called Quantized TT format \cite{khor-qtt-2011}:
we reshape additionally all tensors to the sizes $2 \times 2 \times \cdots \times 2$, and apply the $(d \log_2(n) + \log_2(N_t))$-dimensional TT decomposition, but with each mode size reduced to $2$.

However, the matrix is strongly nonsymmetric, which makes difficulties for the DMRG approach.
We fix the truncation tolerance for the solution to $\mathtt{tol}=10^{-5}$,
and track the Frobenius-norm error of the DMRG solution w.r.t. the reference one, obtained by the AMEn+SVD method with tolerance $10^{-8}$, versus the dimension $d$, see Fig. \ref{fig:casc_dmrg}.
Since the DMRG technique takes into account only local information on the system, its accuracy deteriorates rapidly with the increasing dimension.
A stagnation in a local minimum is also reflected by a sharp drop of the CPU time, since the method skips the ``converged'' TT blocks.
This makes the DMRG unreliable for high--dimensional problems, even if the QTT format allows to get rid of large mode sizes.

Now, we fix the dimension $d=20$, and compare both the error and residual accuracies of all methods, as well as the computational times.
In all cases, the Frobenius-norm tolerance was set to $\mathtt{tol}=10^{-5}$, and the enrichment rank to $\rho=4$.

First of all, since our methods are proven to converge in the SPD case, we shall examine both the initial and symmetrized systems (Fig. \ref{fig:casc_amen_symm}).
A well-known way to treat a general problem via a symmetric method is the normal, or symmetrized formulation, $A^\trans A x = A^\trans y$.
However, both the condition number and the TT ranks of $A^\trans A$ are the squared ones of $A$, and this approach should be avoided when possible.

Three particular techniques are considered: the DMRG method, the AMEn+SVD (marked as ``amen'' in Fig. \ref{fig:casc_amen_symm}) and the $\als(t+z)$ one.
The symmetrized versions are denoted by the ``-s'' tag.

In addition, note that the convergence of the methods may be checked locally due to the zero total correction after the enrichment in Alg. \ref{alg:amen}: before recomputing the $(k+1)$-th block, calculate the local residual provided by the previous solution $T^{(k+1)}$.
If it is below the threshold $\mathtt{tol}$ for all $k$, the method may be considered as converged, and stopped.
Occurrences of this fact are marked by red rectangles (``stop'').

We observe that the symmetrization allows the DMRG method to converge at least to the accuracy $10^{-3}$,
but increases the CPU time by a factor greater than 100 due to the squaring of the TT ranks and condition number of the matrix.
Contrarily, for the AMEn and $\als(t+z)$ methods the symmetrization is completely inefficient and redundant: despite pessimistic theoretical estimates, the nonsymmetric algorithms converge rapidly to an accurate solution approximation.

Though the non-symmetrized methods may admit oscillations in the residual,
the Frobenius-norm error threshold is almost satisfied in both AMEn and $\als(t+z)$ methods.
Nevertheless, the AMEn algorithm appears to be more accurate thanks to the enrichment update in each step.
Also, its local stopping criterion is trustful: it fires just after the real error becomes smaller than the tolerance, which is not the case for other methods.

Since both AMEn-type methods in this test exploit the SVD-based residual approximation,
they demonstrate almost the same CPU times.
However, using the additional techniques from Section \ref{sec:prac} we can reduce the complexity while maintaining almost the same accuracy, see Fig. \ref{fig:casc_var_amen}.
While the AMEn+Chol method still operates with the exact residual, the AMEn+ALS only needs to compute scalar products of the true residual and its low-rank approximation, which makes it more efficient than the AMEn+SVD method, as well as the $\als(t+z)$ one.

Finally, we test the performance of the two AMEn realizations with respect to the enrichment rank $\rho$ (TT-rank of $\tilde z$), see Fig. \ref{fig:casc_amr_rho}.
As expected, the higher $\rho$ is, the more accurate solution can be computed.
On the other hand, it is not necessary to pick very large ranks, since the corresponding accuracy improvement does not overcome the significant increase in CPU time.

\newpage
\subsection{Fokker-Planck equation for complex fluid dynamics}
\begin{figure}[h]
 \begin{center}
  \input{./tikz/polymer.tikz}
 \end{center}
\caption{Bead-spring model of a polymer in a fluid}
\label{fig:beadspring}
\end{figure}
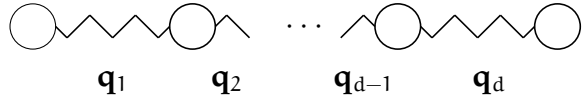
Another example of high-dimensional problems arising in the context of probability distribution modeling, is the Fokker-Planck equation (see e.g. \cite{risken-fpe-1989}).
As a particular application, consider the 8-dimensional Fokker-Planck equation of the polymer micro-model arising in the non-Newtonian fluid dynamics \cite{Lozinski-dilute-2004, Lozinski-2D_FENE-2003,Ammar-greedy_fluid_model-2006}.
The polymer molecules in a solution are subject to the Brownian motion, and are often modeled as bead-spring chains (see Fig. \ref{fig:beadspring}).
The spring extensions, being the degrees of freedom of the dynamical system, become the coordinates in the Fokker--Planck equation.

\begin{figure}[t]
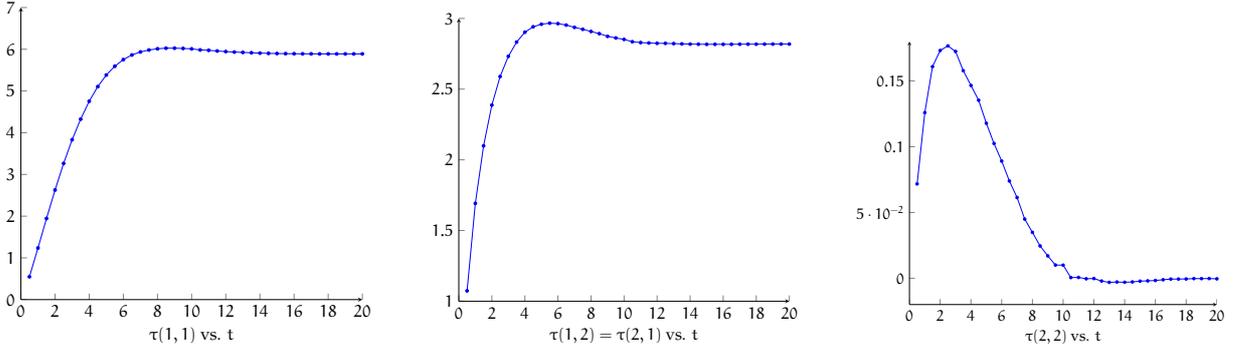

\centering
\resizebox{0.32\textwidth}{!}{\input{./Pic/pgfart.sty} \def\fpplt{1} \input{./Pic/fp8_tts.tikz}} \hfill
\resizebox{0.32\textwidth}{!}{\input{./Pic/pgfart.sty} \def\fpplt{2} \input{./Pic/fp8_tts.tikz}} \hfill
\resizebox{0.32\textwidth}{!}{\input{./Pic/pgfart.sty} \def\fpplt{4} \input{./Pic/fp8_tts.tikz}}
\caption{Components of the stress tensor vs. time}
\label{fig:stress}
\end{figure}
\begin{figure}[t]
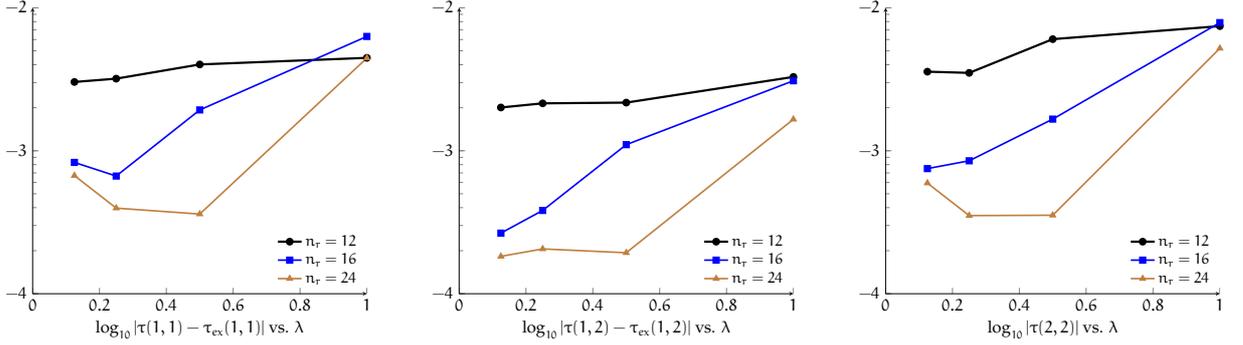

\centering
\resizebox{0.32\textwidth}{!}{\input{./Pic/pgfart.sty} \def\fpplt{2} \input{./Pic/fp8.tikz}} \hfill
\resizebox{0.32\textwidth}{!}{\input{./Pic/pgfart.sty} \def\fpplt{3} \input{./Pic/fp8.tikz}} \hfill
\resizebox{0.32\textwidth}{!}{\input{./Pic/pgfart.sty} \def\fpplt{4} \input{./Pic/fp8.tikz}}
\caption{Accuracy of the stress tensor vs. time step and grid size}
\label{fig:stress_acc}
\end{figure}

We consider the case of 4 two-dimensional \emph{finitely extensible nonlinear elastic} (FENE) springs in the shear flow regime according to \cite{Ammar-greedy_fluid_model-2006},
\begin{equation}
\frac{\partial \psi (\mathbf{q},t)}{\partial t} = \sum\limits_{i=1}^{d} \nabla_{\mathbf{q}_i} \cdot \left(\mathbf{K} \mathbf{q} \psi(\mathbf{q},t) -  \sum\limits_{j=1}^{d} D_{ij} \left(\mathbf{F}_j(\mathbf{q}) \psi(\mathbf{q},t) - \nabla_{\mathbf{q}_j} \psi(\mathbf{q},t)\right) \right),
\label{eq:fpe}
\end{equation}
where $\mathbf{q} = (\mathbf{q}_1,...,\mathbf{q}_8) = (q_{1,1}, q_{1,2},...,q_{4,2})$ is the stacked spring extension vectors ($q_{p,k}$ is the displacement of the $p$-th spring in the $k$-th direction),
$$
D =  \frac{1}{4}\mathrm{tridiag}\{-1,2,-1\} \otimes I_2 =  \frac{1}{4} \begin{bmatrix}2 & -1 \\ -1 & 2 & -1 \\ & -1 & 2 & -1 \\ & & -1 & 2\end{bmatrix}  \otimes \begin{bmatrix}1 \\ & 1\end{bmatrix}
$$
is a spring interaction tensor,
$$
\mathbf{K} = I_4 \otimes (\nabla_x\mathbf{u}) = \mathrm{diag}\{1,1,1,1\}\otimes \begin{bmatrix}0 & 0.8 \\ 0 & 0\end{bmatrix}
$$
is a flow velocity gradient (shear flow case),
and
$$
\mathbf{F}_j(\mathbf{q}) = \dfrac{\mathbf{q}_j}{1-|\mathbf{q}_j|^2/b}, \quad |\mathbf{q}_j|^2 = q_{j,1}^2+q_{j,2}^2, \quad b=5,
$$
is the FENE spring force.
Note that the singularity in $\mathbf{F}_j$ limits the maximal length of a spring to $\sqrt{b}$.
Moreover, the probability density $\psi$ at the point $|\mathbf{q}_j|=\sqrt{b}$ (and any with larger modulus) is zero.
Therefore, the domain shrinks to the product of balls $B=\mathbb{B}_{\sqrt{b}}^{\otimes d}.$

\begin{figure}[t]
\centering
\resizebox{0.4\textwidth}{!}{\input{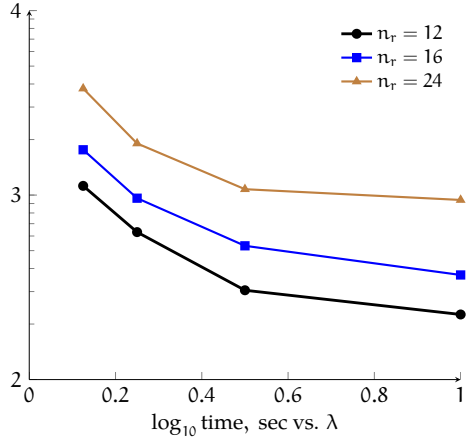} \def\fpplt{0} \input{./Pic/fp8.tikz}}
\caption{CPU time vs. time step and grid size}
\label{fig:fpe_ttimes}
\end{figure}

A quantity of interest is the average polymeric contribution to the stress tensor,
\begin{equation}
\tau(t) = \sum\limits_{i=1}^d \left\langle \mathbf{q}_i \mathbf{F}_i(\mathbf{q})^{\top} - I_2 \right\rangle = \sum\limits_{i=1}^d \int_B \psi(\mathbf{q},t) \mathbf{q}_i \mathbf{F}_i(\mathbf{q})^{\top} d \mathbf{q} - d\cdot I_2,
\label{eq:stress}
\end{equation}
with the normalization assumption $\int \psi(\mathbf{q},t) d \mathbf{q} = 1$.

To recast the problem domain into a hypercube, the polar coordinates are employed,
$\mathbf{q} \rightarrow (r_1,\theta_1,...,r_4,\theta_4) \in \left([0,\sqrt{b}) \otimes [0,2\pi)\right)^{\otimes 4}$.
The discretization is done via the spectral elements method (see e.g. \cite{trefethen-spectral-2000}).
We will vary the number of spectral elements in each radial direction $n_r$, but the number of angular elements (in $\theta_i$) is fixed to $2 n_r$.
With typical values $n_r \sim 20$, we end up with tensors of size $\sim 10^{12}$ and dense populated matrices, which are intractable in the full format.
Since the spectral differentiation matrices are found to be incompressible in the QTT format, the 8-dimensional TT representation is used.

We would like to compute the stationary state of \eqref{eq:fpe}, so we use the simple implicit Euler (inverse power) method as the time discretization,
$$
(M+\lambda A) \psi(t_{m+1}) = M \psi(t_m), \quad t_m = \lambda m, ~~m=0,...,N_t-1,
$$
where $M$ is the mass matrix, $A$ is the stiffness matrix.
The time integration was performed until $T = \lambda N_t = 20$, which is enough to approximate the steady state with a satisfactory accuracy, and the (unnormalized) initial state was chosen $\psi(\mathbf{q},0)=\bigotimes_{i=1}^4 (1-|\mathbf{q}_i|^2/b)^{b/2}$, which corresponds to the zero velocity gradient $\mathbf{K}=0$.
Since $\lambda$ is not a ``time step'' but a parameter of the inverse power method, we will check the performance w.r.t. $\lambda$ as well.

In the previous example we have observed that the AMEn+SVD method is in fact superfluous, since the AMEn+ALS method delivers the same accuracy with lower cost.
Both mode sizes (up to $48$) and TT--ranks (up to $73$) in this example are relatively large,
so we will consider only the AMEn+ALS.
We set the Frobenius-norm threshold to $\mathtt{tol}=10^{-4}$, and the enrichment rank $\rho=3$.
The initial guess for the AMEn+ALS method is taken from the previous Euler step.

First, let us track the evolution of the stress tensor components \eqref{eq:stress} versus Euler iterations, see Fig. \ref{fig:stress}.
We see that the stress does really stabilize in the chosen time range.
Moreover, the last component tends to zero, and can therefore be used as an in-hand measure of the accuracy.
In addition, we compare $\tau(1,1)$ and $\tau(1,2)$ with the reference values computed with $n_r=28$ and $\mathtt{tol}=10^{-5}$, see Fig. \ref{fig:stress_acc}.
For all $\lambda$ except $1$ (which is too large), and $n_r=24$, the accuracy attained is of the order $10^{-4} \div 10^{-3}$.
Note that  typical accuracies of greedy or MC methods for many-spring models are of the order $10^{-1}$ \cite{Ammar-greedy_fluid_model-2006,Venkiteswaran-QMC_FP-2005}.

Finally, the computational times can be seen in Fig. \ref{fig:fpe_ttimes}.
As expected, the complexity increases quadratically with the number of spectral elements $n_r$.
An interesting feature is that the total CPU time decays with increasing $\lambda$.
It points out that the performance of the AMEn method depends weakly on $\lambda$, and henceforth on the matrix spectrum.  
On the contrary, the quality of the initial guess (in terms of both ranks and accuracy) is crucial.
This may motivate attempts to relate the AMEn methods to Newton or Krylov iterations in a future research.

\newpage
\section{Conclusion}\label{sec:conclusion}
In this paper we develop a new version of the fast rank--adaptive solver for tensor--structured symmetric positive definite linear systems in higher dimensions.
Similarly to the algorithms from~\cite{ds-amr1-2013}, the proposed AMEn method combines the one-dimensional local updates with the steps where the basis is expanded using the information about the global residual of the high--dimensional problem.
However, in AMEn the same steps are ordered in such a way that only one or two neighboring cores are modified at once.
Both methods from~\cite{ds-amr1-2013} and the AMEn converge globally, and the convergence rate is established w.r.t. the one of the steepest descent algorithm.
The practical convergence in the numerical experiments is significantly faster than the theoretical estimate.
The AMEn algorithm appears to be more accurate in practical computations than the previously known methods, especially if local problems are solved roughly.

The asymptotic complexity of the AMEn is linear in the dimension and mode size, similarly to the algorithms from~\cite{ds-amr1-2013}.
The complexity w.r.t. the rank parameter is sufficiently improved taking into the account that a limiting step is the approximation of the residual, where the high accuracy is not always essential for the convergence of the whole method.
We propose several cheaper alternatives to the SVD-based TT-approximation, namely the Cholesky decomposition and the inner ALS algorithm.
The ALS approach provides a significant speedup, while maintaining almost the same convergence of the algorithm.

Finally, we apply the developed AMEn algorithm to general (non-SPD) systems, which arise from high--dimensional Fokker--Planck and chemical master equations.
Theoretical convergence analysis can be made similarly to the FOM method, which is rather pessimistic and puts very strong requirements on the matrix spectrum.
In numerical experiments we observe a surprisingly fast convergence, even for strongly non--symmetric systems.
Here the AMEn demonstrates a significant advantage over the DMRG technique, which is known to stagnate, especially in high dimensions, see~\cite{Os-mvk2-2011,dc-tt_gmres-2011} and Fig.~\ref{fig:casc_amen_symm}.

There are many directions of a further research based on the ideas of~\cite{ds-amr1-2013} and  this paper.

First, the ideas developed in this paper can be generalized to other problems, e.g. finding the ground state of a many-body quantum system or a particular state close to a prescribed energy.
The combination of update and basis enrichment steps looks very promising for a wide class of problems, as soon as the corresponding classical iterative algorithms can be adapted to provide a proper basis expansion in higher dimensions.
A huge work is done in the community of greedy approximation methods, where the cornerstone is a subsequent rank-one update of the solution.

Second, there is a certain mismatch between the theoretical convergence estimates, which are at the level of the one--step steepest descent algorithm, and the practical convergence pattern, which looks more like the one of the GMRES.
This indicates that there are further possibilities to improve our understanding of the convergence of the AMEn and similar methods.
Our rates can benefit from sharp estimates of the progress of the one-dimensional update steps, which at the moment are available only in a small vicinity of a true solution, which is hard to satisfy in practice, see~\cite{ushmaev-tt-2013}.
The superlinear convergence observed in numerical experiments inspires us to look for possible connections with the theory of Krylov--type iterative methods and a family of Newton methods.

Finally, we look forward to solving more high--dimensional problems, and are sure that they will bring new understanding of the advantages and drawbacks of the proposed method, and new questions and directions for a future research.

\newpage

\newpage
\appendix
\section{FOM theory}

\def\ang#1#2{(\widehat{{#1},{\:}{#2}})}

As was observed in the numerical experiments, the AMEn method works successfully even being applied directly to non-symmetric systems.
Though we cannot support this behavior with sharp estimates, one may proceed similarly to Section \ref{sec:amen}, and establish a formal theory, relating the AMEn to the Full Orthogonalization Method.

\subsection{Galerkin projection and angles between subspaces}
Like in the SPD case, we begin the analysis from the two-dimensional case.
Given a linear system $Ax=y$ and some basis $V$, the projection method is performed as follows,
\begin{equation}\label{eq:gal_oblique}
x = V w = V (V^\trans A V)^{-1} V^\trans y, \qquad y-Ax = (I-AV (V^\trans A V)^{-1} V^\trans)y.
\end{equation}
Given an initial guess $t$, we assume $t \in \mathrm{span}(V)$, and $z=y-At \in \mathrm{span}(V)$.
Then it holds also $y-Ax = (I-AV (V^\trans A V)^{-1} V^\trans)z$.

So, \eqref{eq:gal_oblique} performs an oblique projection of the residual.
Its analysis is often conducted with the help of the orthogonal projection,
$$
((AV)^\trans AV)w = (AV)^\trans z, \qquad y-Ax = (I-AV(V^\trans A^\trans A V)^{-1} (AV)^\trans)z = (I-P_{AV})z,
$$
i.e. the residual minimization on $V$.
The case $V=z$ is known as the MINRES method. Its convergence was analysed in e.g. \cite{saad-iter},
$$
\|y-Ax\| = \sin\ang{Az}{z} \|z\|, \quad \cos\ang{Az}{z} = \frac{|(Az,z)|}{\|Az\| \|z\|},
$$
i.e. $\ang{Az}{z}$ is the acute angle between $Az$ and $z$.
The worst convergence rate is estimated as
$$
\omega_{MR} = \max\limits_{z} \sin\ang{Az}{z}, \quad \sqrt{1-\omega_{MR}^2} = \min\limits_{z \neq 0}\frac{|(Az,z)|}{\|Az\| \|z\|},
$$
and for a positive definite matrix is guaranteed to be less than 1.
The same approach may be used for the block case as well,
$$
\|y-Ax\| = \sin\ang{AV}{z} \|z\|, \quad \cos\ang{AV}{z} = \max_{q \neq 0}\frac{|(AVq, z)|}{\|AVq\| \|z\|}.
$$
Obviously, if $z \in \mathrm{span}(V)$, it holds $\sin\ang{AV}{z} \le \sin\ang{Az}{z}<1$.

Unfortunately, for the oblique projection \eqref{eq:gal_oblique} one cannot guarantee the monotonous convergence in general.
However, assuming a certain well-conditioning of the system, we may relate the old and new residuals by a factor smaller than 1 as well.
\begin{lemma}\label{lem:fom}
Given a column-orthogonal matrix $V$, initial guess $t \in \mathrm{span}(V)$.
Assume the smallest eigenvalue $\mu = \lambda_{\min}(V^\trans A V+V^\trans A^\trans V)/2>0$,
and $\|z-VV^\trans z\| \le \eps \|z\|$.
Denote $\frac{\mu}{\|AV\|} = \sqrt{1-\omega_V^2}$.
Then, the progress of \eqref{eq:gal_oblique} is bounded by
$$
\|y-Ax\| \le \left(\eps + \frac{\omega_V}{\sqrt{1-\omega_V^2}} \sqrt{1-\eps^2}\right)\|z\|.
$$
\end{lemma}
\begin{proof}
First of all, notice that $y-Ax$ is orthogonal to $V$,
$$
\begin{array}{rcl}
y-Ax & = & (VV^\trans)(I-AV (V^\trans A V)^{-1} V^\trans)z + (I-VV^\trans)(I-AV (V^\trans A V)^{-1} V^\trans)z \\
& = & (I-VV^\trans)z - (I-VV^\trans)AV (V^\trans A V)^{-1} V^\trans z.
\end{array}
$$
Then, $\|y-Ax\| \le \eps\|z\| + \sin\ang{V}{AVw} \|AVw\|$, where $w = (V^\trans A V)^{-1} V^\trans z$.
For the angle we can derive the following chain of inequalities,
\begin{equation}\label{eq:angles}
\begin{array}{rcl}
\cos\ang{V}{AVw} & = & \max\limits_{\|q\|=1} \frac{|(Vq, AVw)|}{\|AVw\|} \ge
\min\limits_{w \neq 0} \frac{\|V^\trans A V w\|}{\|AVw\|} \\
& \ge & \min\limits_{\|w\|=1}\frac{\|w^\trans V^\trans A V w\|}{\|AVw\|} \ge  \min\limits_{\|w\|=1}\frac{\|w^\trans V^\trans A V w\|}{\|AV\|},
\end{array}
\end{equation}
from which we get $\sin\ang{V}{AVw} \le \omega_V$.
On the other hand,
\begin{equation}\nonumber
 \min\limits_{\|w\|=1}\frac{\|w^\trans V^\trans A V w\|}{\|AV\|} = \frac{\lambda_{\min}((V^\trans A V+V^\trans A^\trans V)/2)}{\|AV\|} \le \frac{\sigma_{\min}(V^\trans A V)}{\|AV\|},
\end{equation}
so that $\frac{\|AV\|}{\sigma_{\min}(V^\trans A V)} \le \frac{1}{\sqrt{1-\omega_V^2}}$.
Therefore, the residual estimates as follows,
$$
\sin\ang{V}{AVw} \|AVw\| \le \omega_V \|AV\| \frac{1}{\sigma_{\min}(V^\trans A V)} \|V^\trans z\| \le \frac{\omega_V}{\sqrt{1-\omega_V^2}} \sqrt{1-\eps^2} \|z\|.
$$
\end{proof}
\begin{remark}
It holds
$$
\sqrt{1-\omega_{MR}^2} = \min\limits_{\|z\|=1}\frac{|z^\trans Az|}{\|Az\|} \le \min\limits_{\|w\|=1} \frac{\|w^\trans V^\trans A V w\|}{\|AVw\|} \le \cos\ang{V}{AVw},
$$
since the minimization over $Vw$ is a restriction w.r.t. the minimization over $z$ in the full space.
Hence, $\sin\ang{V}{AVw} \le \omega_{MR}$.
However, $\sin\ang{V}{AVw}/\sqrt{1-\omega_V^2} \ge \tan\ang{V}{AVw}$ might be greater than $\omega_{MR}$, and even greater than 1.
\end{remark}
\begin{remark}
If $V$ contains the $m$-th Krylov subspace, we obtain the so-called FOM method.
The progress of the FOM can be related to that of the GMRES as follows \cite{saad-iter},
$$
\omega_m^F = \frac{\omega_m^G}{\sqrt{1-\left(\omega_m^G/\omega_{m-1}^G\right)^2}},
$$
where $\omega^G_k$, $\omega^F_k$ are the progresses of the $k$-step GMRES and FOM, resp.
Note the similar term $\omega_V/\sqrt{1-\omega_V^2}$ in Lemma \ref{lem:fom}.
\end{remark}
\begin{remark}
The condition $\|z-VV^\trans z\| \le \eps \|z\|$ may reflect the residual approximation, i.e. $z \approx \tilde z \in \mathrm{span}(V)$, but $z \notin \mathrm{span}(V)$.
Both SVD- and ALS-based approximations (see Section \ref{sec:prac}) fit to this scheme: the SVD approximation reads $\tilde z = \tilde U \tilde U^\trans z$, where $\tilde U$ is the singular vectors, and the ALS approximation reads $\tilde z = \Z_{\neq k} \Z_{\neq k}^\trans z$.
\end{remark}

\subsection{Recurrent residual accumulation}
Lemma \ref{lem:fom} applies immediately to the two-dimensional AMEn method, by setting $V = \X_1$.
Despite the generally pessimistic estimate, it occurs in practice that $\X_1^\trans A \X_1$ is nonsingular, and moreover, $\omega_{\X_1}$ is rather small such that $\|y-Ax\| < \|z\|$ and converges rapidly.

A nice property of Theorem \ref{thm:amen} is that it itself does not rely on a particular form of $\X_k$.
We only needed that the Galerkin conditions $A_{k+1}x^{>k}=y_{k+1}$ make the error $\|x^{\geq k}_\tru-x^{\geq k}\|_{A_k}$ strictly smaller than $\|x^{\geq k}_\tru-u^{\geq k}\|_{A_k}$.

Here, we write the similar result in terms of residuals.
\begin{lemma}\label{lem:amr}
Suppose in the $k$-th step of the multidimensional AMEn method, the ALS step provides the residual decrease
$$
\|y_k - A_k u^{\geq k}\| = \mu_k \|y_k-A_k t^{\geq k}\|, \qquad \mu_k \le 1,
$$
and the exact computation of the rest cores $x^{>k}_\tru = A_{k+1}^{-1} y_{k+1}$ after the enrichment provides the residual decrease
$$
\|y_k - A_k x^{\geq k}\| = \omega_{k}  \|y_k-A_k u^{\geq k}\|, \qquad \omega_k<1,
$$
where $\omega_k \leq \omega_{\X_k}$ from Lemma \ref{lem:fom} with $\X_k=\P_1(X^{(k)},\ldots,X^{(d)})$.
Then, the total convergence rate of the AMEn method is bounded by
$$
\|y-Ax\| \le \sum\limits_{k=1}^{d-1} \omega_k \mu_k \prod\limits_{m=1}^{k-1} \frac{\mu_m}{\sqrt{1-\omega_m^2}} \cdot \|y-At\|.
$$
\end{lemma}
\begin{proof}
As previously, we assume that $X^{(d)}$ is computed exactly.
Then,
$$
\|y_{d-1} - A_{d-1}x^{\geq d-1}\| = \omega_{d-1}  \|y_{d-1} - A_{d-1} u^{\geq d-1}\|.
$$
The base of the recursion is proved.

Suppose the theorem holds for $A_2 x^{\geq 2} = y_2$, i.e.
\begin{equation}\label{eq:amr_res_rec}
\|y_{2} - A_2 x^{\geq 2}\| \le \sum\limits_{k=2}^{d-1} \omega_k \mu_k \prod\limits_{m=2}^{k-1} \frac{\mu_m}{\sqrt{1-\omega_m^2}} \cdot \|y_{2} - A_2 t^{\geq 2}\| = \Omega \|y_{2} - A_2 t^{\geq 2}\|,
\end{equation}
and write the total progress for the whole system. We have
\begin{equation}
y - A\tau(X^{(1)}, X^{\ge 2}) = y - A\tau(X^{(1)}, X^{\ge 2}_\tru) + A\tau(X^{(1)}, X^{\ge 2}_\tru) - A\tau(X^{(1)}, X^{\ge 2}).
\label{eq:amr-d-newres}
\end{equation}

The exact solution for the second block is the oblique projection \eqref{eq:gal_oblique}, hence
$$
y - A\tau(X^{(1)}, X^{\ge 2}_\tru) = (I-A \X_1(\X_1^{\trans} A \X_1)^{-1} \X_1^{\trans}) y.
$$
The last two terms in \eqref{eq:amr-d-newres} are similar to that in Theorem \ref{thm:amen},
$$
A\tau(X^{(1)}, X^{\ge 2}_\tru) - A\tau(X^{(1)}, X^{\ge 2}) = A\X_1 (x^{\ge 2}_\tru - x^{\ge 2}),
$$
but now it is not orthogonal to $(I-A \X_1(\X_1^{\trans} A \X_1)^{-1} \X_1^{\trans})$.
Therefore, we can only use the triangle inequality,
$$
\|y - A\tau(X^{(1)}, X^{\ge 2})\| \le \|(I-A \X_1(\X_1^{\trans} A \X_1)^{-1} \X_1^{\trans}) y\| + \|A\X_1 (x^{\ge 2}_\tru - x^{\ge 2})\|.
$$
The first term is the residual after the Galerkin solution, which is bounded by $\omega_1 \|y-Au\|$.
For the second term, we have the recursion assumption \eqref{eq:amr_res_rec}, that is
$$
\|\X_1^\trans A \X_1 (x^{\ge 2}_\tru - x^{\ge 2})\| \le \Omega \|\X_1^\trans A \X_1 (x^{\ge 2}_\tru - t^{\ge 2})\|.
$$
However, the only way to relate $\|A w\|$ and $\|\X_1^\trans A w\|$ is to use the angle between $\X_1$ and $A\X_1$, employing \eqref{eq:angles},
$$
\|\X_1^\trans A \X_1 (x^{\ge 2}_\tru - x^{\ge 2})\| \ge \sqrt{1-\omega_1^2} \|A \X_1 (x^{\ge 2}_\tru - x^{\ge 2})\|,
$$
$$
\|A \X_1 (x^{\ge 2}_\tru - x^{\ge 2})\| \le \frac{\Omega}{\sqrt{1-\omega_1^2}} \|\X_1^\trans A \X_1 (x^{\ge 2}_\tru - t^{\ge 2})\|.
$$
Since $\X_1^\trans A\X_1 x^{\ge 2}_\tru = \X_1^\trans y$, and $\X_1 t^{\ge 2} = u$, it holds
$$
\|A \X_1 (x^{\ge 2}_\tru - x^{\ge 2})\| \le \frac{\Omega}{\sqrt{1-\omega_1^2}} \|\X_1^\trans z\| \le \frac{\Omega}{\sqrt{1-\omega_1^2}} \|z\|.
$$
Therefore, for the total residual we have
$$
\|y-A\tau(X^{(1)}, X^{\ge 2})\| \le \left(\omega_1 + \frac{1}{\sqrt{1-\omega_1^2}} \Omega\right)\|y-Au\|.
$$
Plugging in the ALS update, the final estimate for \eqref{eq:amr-d-newres} now writes as follows,
$$
\|y-A\tau(X^{(1)}, X^{\ge 2})\| \le \mu_1\left(\omega_1 + \frac{1}{\sqrt{1-\omega_1^2}} \Omega\right)\|y-At\|,
$$
which finishes the recursion.
\end{proof}

Contrarily to the symmetric positive definite case, where the total progress of the AMEn method was deteriorating with $d$, but less than 1 in any case,
here we may have a situation when the progress bound given by Lemma \ref{lem:amr} is greater than 1.
Up to this moment, the only available estimate is $\omega_k \le \omega^F_1$, since we enrich the basis by $Z^{(k)}_k$, i.e. the first Krylov vector only.
In principle, it is possible to include a larger approximate Krylov basis into the enrichment, i.e.
$$
\begin{bmatrix}Z^{(k)}_k & Q^{[1](k)} & \cdots & Q^{[m-1](k)}\end{bmatrix},
$$
where $q^{[p]} = \tau(Q^{[p](k)}, \ldots, Q^{[p](d)}) \approx A_k^p z_k$, $p=1,\ldots,m-1$.
However, this was not found to be reasonable in practical experiments.
In all considered cases, the decays $\omega_k$ and $\mu_k$ provided by the single enrichment $Z_k^{(k)}$ appeared to be sufficiently small to ensure the convergence, fast enough to overcome the work required to prepare several Krylov vectors.

\end{document}

%% file: tikz/alsamen.tikz
  \begin{tikzpicture}[x=15mm,y=-12mm]
    \filldraw[very thin,draw=white!00,fill=green!02]  (0,-.5) circle[x radius=2, y radius=1];
    \filldraw[very thin,rounded corners=10,draw=white!00,fill=blue!02]   (-6,.3) rectangle (-3,6.3);
    \filldraw[very thin,rounded corners=10,draw=white!00,fill=blue!04]   (-3,.3) rectangle (-0,6.3);
    \filldraw[very thin,rounded corners=10,draw=white!00,fill=red!02]    (+3,.3) rectangle (+0,6.3);
    \filldraw[very thin,rounded corners=10,draw=white!00,fill=red!04]    (+6,.3) rectangle (+3,6.3);
    \node (init)  at (0,-1) {initial guess};
    \node (als)  at (-3.0,0.0)   {\textcolor{blue}{ALS(t+z)}};
    \node (amen) at (+3.0,0.0)   {\textcolor{red}{AMEn}};
    \node (alsup)   at (-4.2,6)  {\textcolor{blue}{update}};
    \node (alsrich)  at (-1.5,6) {\textcolor{blue}{expand}};
    \node (amenup)   at (1.5,6)  {\textcolor{red}{update}};
    \node (amenich)  at (4.5,6)  {\textcolor{red}{expand}};
    \node[rectangle] (0)     at (0,-0.3)
        {$ \left[\begin{array}{cc} t & \phantom{z}\end{array}\right]
          \left[\begin{array}{cc} t & \\ & \phantom{z}\end{array}\right]
          \left[\begin{array}{c} t  \\  \phantom{z}\end{array}\right]   $};
    \node[rectangle] (alsz)  at (-1.5,1)
        {$ \left[\begin{array}{cc} t & z \end{array}\right]
          \left[\begin{array}{cc} t & \\ & z\end{array}\right]
          \left[\begin{array}{c} t  \\ z\end{array}\right]   $};
    \node[rectangle] (als1)  at (-4.5,2)
        {$ \left[\begin{array}{cc} x & x\end{array}\right]
          \left[\begin{array}{cc} t & \\ & z\end{array}\right]
          \left[\begin{array}{c} t  \\  z\end{array}\right]   $};
    \node[rectangle] (als2)  at (-4.5,3.5)
        {$ \left[\begin{array}{cc} x & x \end{array}\right]
          \left[\begin{array}{cc} x & x  \\ x  & x \end{array}\right]
          \left[\begin{array}{c} t  \\  z\end{array}\right]   $};
    \node[rectangle] (als3)  at (-4.5,5)
        {$ \left[\begin{array}{cc} x & x\end{array}\right]
          \left[\begin{array}{cc} x & x \\ x & x\end{array}\right]
          \left[\begin{array}{c} x  \\  x \end{array}\right]   $};
    \node[rectangle] (amen1x) at (+1.5,1)
        {$ \left[\begin{array}{cc} u & \phantom{z}\end{array}\right]
          \left[\begin{array}{cc} t & \\ & \phantom{z}\end{array}\right]
          \left[\begin{array}{c} t  \\  \phantom{z}\end{array}\right]   $};
    \node[rectangle] (amen1z) at (+4.5,2)
        {$ \left[\begin{array}{cc} u & z\end{array}\right]
          \left[\begin{array}{cc} t & \\ & \phantom{z}\end{array}\right]
          \left[\begin{array}{c} t  \\  \phantom{z}\end{array}\right]   $};
    \node[rectangle] (amen2x) at (+1.5,3)
        {$ \left[\begin{array}{cc} u & z\end{array}\right]
          \left[\begin{array}{cc} u &  \\ u & \phantom{z}\end{array}\right]
          \left[\begin{array}{c} t  \\  \phantom{z}\end{array}\right]   $};
    \node[rectangle] (amen2z) at (+4.5,4)
        {$ \left[\begin{array}{cc} u & z\end{array}\right]
          \left[\begin{array}{cc} u & z \\ u & z\end{array}\right]
          \left[\begin{array}{c} t  \\  \phantom{z}\end{array}\right]   $};
    \node[rectangle] (amen3x) at (1.5,5)
        {$ \left[\begin{array}{cc} u & z\end{array}\right]
          \left[\begin{array}{cc} u & z \\ u & z\end{array}\right]
          \left[\begin{array}{c} u  \\ u \end{array}\right]   $};

   \draw[->,blue,thick] (0.west)+(left:0.1mm) to[out=left,in=up] (alsz.north);
   \draw[<-,blue,thick] (als1.north east)+(left:3ex) to[out=up,in=left] (alsz.west);
   \draw[->,blue,thick] (als1.south)+(down:0.1mm) to[out=down,in=up] (als2.north);
   \draw[->,blue,thick] (als2.south)+(down:0.1mm) to[out=down,in=up] (als3.north);

   \draw[->,red,thick] (0.east)+(right:0.1mm) to[out=right,in=up] (amen1x.north);
   \draw[<-,red,thick] (amen1z.north west)+(right:3ex) to[out=up,in=right] (amen1x.east);
   \draw[<-,red,thick] (amen2x.north east)+(left:3ex) to[out=up,in=left] (amen1z.west);
   \draw[<-,red,thick] (amen2z.north west)+(right:3ex) to[out=up,in=right] (amen2x.east);
   \draw[<-,red,thick] (amen3x.north east)+(left:3ex) to[out=up,in=left] (amen2z.west);
  \end{tikzpicture}
 

%% file: tikz/amen2.tikz
  \begin{tikzpicture}[x=18mm,y=-8mm]

    \filldraw[thin,draw=white!00,fill=green!03]  (-1,-2) rectangle (1,3);
    \filldraw[thin,draw=white!00,fill=red!03]    ( 1,-2) rectangle (3,3);
    \filldraw[thin,draw=white!00,fill=red!05]    ( 3,-2) rectangle (5,3);
    \filldraw[thin,draw=white!00,fill=purple!01] ( 5,-2) rectangle (7,3);
    \filldraw[thin,draw=white!00,fill=red!07]    ( 7,-2) rectangle (9.5,3);
    \node (0t) at (0,-1.4) {initial guess};
    \node (1t) at (2,-1.4) {update core 1};
    \node (2t) at (4,-1.4) {expand basis};
    \node (3t) at (6,-1.4) {Galerkin correction};
    \node (4t) at (8.5,-1.4) {update core 2};
    \node[rectangle] (0) at (0,0)
        {$ \left[\begin{array}{cc} t & \phantom{z}\end{array}\right]
          \left[\begin{array}{c} t  \\  \phantom{z}\end{array}\right]   $};
    \node[rectangle] (1) at (2,0)
        {$ \left[\begin{array}{cc} u & \phantom{z}\end{array}\right]
          \left[\begin{array}{c} t  \\  \phantom{z}\end{array}\right]   $};
    \node[rectangle] (2) at (4,0)
        {$ \left[\begin{array}{cc} u & z\end{array}\right]
          \left[\begin{array}{c} t  \\  \phantom{z}\end{array}\right]   $};
    \node[rectangle] (3) at (6,1)
        {$ \left[\begin{array}{cc} u & z\end{array}\right]
          \left[\begin{array}{c} t  \\ v \end{array}\right]   $};
    \node[rectangle] (4) at (8.5,0)
        {$ \left[\begin{array}{cc} u & z\end{array}\right]
          \left[\begin{array}{c}  x \\ x \end{array}\right]   $};
    \node[rectangle] (5) at (7.8,2)
        {$ \left[\begin{array}{cc} u & z\end{array}\right]
          \left[\begin{array}{c}  \tilde x \\ \tilde x \end{array}\right]   $};

   \draw[->,red,thick] (0) to node[above,midway] {$\mu$} (1);
   \draw[->,red,thick] (1) to (2);
   \draw[->,red,thick,dashed] (2) to node[above,midway] {$\omega$} (3);
   \draw[->,red,thick,dashed] (3) to (4);
   \draw[->,red,thick] (2.east)  to[out=45,in=left] node[above,very near end] {exact} (4.west);
   \draw[->,red,thick] (2.south) to[out=down,in=left] node[below,near end] {approximate} (5.west);
  \end{tikzpicture}

%% file: Pic/conv_lp_i_2.tikz
\begin{tikzpicture}

\begin{axis}[%
xmode=normal,ymode=log,
cycle list name=amen,
xlabel={$\log_{10}\mathrm{error}$ vs. iterations},
xmin=0, xmax=5,
ymin=1e-09, ymax=1e-3,
yminorticks=true,
legend style={}]

\foreach \m in {dmrg,alsx,amrsvd,amrals}
  {
  \addplot+[no marks]   table[header=false, x index = 0, y index = 2]{./box/dat/conv_lp\nnn_\m.dat};
  }

\pgfplotsset{cycle list shift=-4}

\foreach \m in {dmrg,alsx,amrsvd,amrals}
  {
  \addplot+[only marks] table[header=false, x index = 0, y index = 2]{./box/dat/conv_lp\nnn_\m.int};
  }

\legend{dmrg,als(t+z),amen+svd,amen+als}

\end{axis}
\end{tikzpicture}%

%% file: Pic/conv_lp_t_2.tikz
\begin{tikzpicture}

\begin{axis}[%
xmode=log,ymode=log,
cycle list name=amen,
xlabel={$\log_{10}\mathrm{error}$ vs. $\log_{10}\mathrm{time,~sec}$ },
xmin=1e-2, xmax=1e3,
xminorticks=true,
ymin=1e-09, ymax=1e-3,
yminorticks=true,
legend style={at={(.99,.99)},anchor=north east}
]

\foreach \m in {dmrg,alsx,amrsvd,amrals}
  {
  \addplot+[no marks]   table[header=false, x index = 1, y index = 2]{./box/dat/conv_lp\nnn_\m.dat};
  }

\pgfplotsset{cycle list shift=-4}

\foreach \m in {dmrg,alsx,amrsvd,amrals}
  {
  \addplot+[only marks] table[header=false, x index = 1, y index = 2]{./box/dat/conv_lp\nnn_\m.int};
  }


\end{axis}
\end{tikzpicture}%

%% file: Pic/resids_cme20_dmrg.tikz
\begin{tikzpicture}

\begin{semilogyaxis}[%
cycle list name=amen,
xmin=2, xmax=20,
ymin=1e-06, ymax=1,
xlabel={$\log_{10}\mathrm{error}$ vs. dimension $d$},
yminorticks=true]

\addplot+[] table[header=true, x=dim, y=resid]{./box/dat/cme20_dmrg_d.dat};

\end{semilogyaxis}
\end{tikzpicture}%

%% file: Pic/ttimes_cme20_dmrg.tikz
\begin{tikzpicture}

\begin{axis}[%
cycle list name=amen,
xmin=2, xmax=20,
ymin=0, ymax=35,
xlabel={time,~sec vs. dimension $d$},
yminorticks=true]

\addplot+[]
table[header=true, x=dim, y=time]{./box/dat/cme20_dmrg_d.dat};

\end{axis}
\end{tikzpicture}%

%% file: Pic/resids_cme20_symm_i.tikz
\begin{tikzpicture}

\begin{semilogyaxis}[%
cycle list name=amensymm,
xmin=0, xmax=20,
ymin=1e-06, ymax=1,
xlabel={$\log_{10}\mathrm{residual}$ vs. iterations},
yminorticks=true]

\addplot+[] table[header=true, x=iter, y=r_dmrg  ]{./box/dat/cme20_conv_symm_rho4.dat};
\addplot+[] table[header=true, x=iter, y=r_dmrg_s]{./box/dat/cme20_conv_symm_rho4.dat};
\addplot+[] table[header=true, x=iter, y=r_alsx  ]{./box/dat/cme20_conv_symm_rho4.dat};
\addplot+[] table[header=true, x=iter, y=r_alsx_s]{./box/dat/cme20_conv_symm_rho4.dat};
\addplot+[] table[header=true, x=iter, y=r_amen  ]{./box/dat/cme20_conv_symm_rho4.dat};
\addplot+[] table[header=true, x=iter, y=r_amen_s]{./box/dat/cme20_conv_symm_rho4.dat};

\addplot+[] table[header=true,x=iter, y=resid]{./box/dat/cme20_stop_symm_rho4.dat};

\end{semilogyaxis}
\end{tikzpicture}%

%% file: Pic/resids_cme20_symm_t.tikz
\begin{tikzpicture}

\begin{loglogaxis}[%
cycle list name=amensymm,
xmin=0.1, xmax=10000,
xminorticks=true,
ymin=1e-06, ymax=1,
yminorticks=true,
xlabel={$\log_{10}\mathrm{residual}$ vs. $\log_{10}\mathrm{time,~sec}$},
legend style={at={(0.03,0.03)},anchor=south west}]

\addplot+[] table[header=true, x=t_dmrg,   y=r_dmrg   ]{./box/dat/cme20_conv_symm_rho4.dat};
\addplot+[] table[header=true, x=t_dmrg_s, y=r_dmrg_s ]{./box/dat/cme20_conv_symm_rho4.dat};
\addplot+[] table[header=true, x=t_alsx,   y=r_alsx   ]{./box/dat/cme20_conv_symm_rho4.dat};
\addplot+[] table[header=true, x=t_alsx_s, y=r_alsx_s ]{./box/dat/cme20_conv_symm_rho4.dat};
\addplot+[] table[header=true, x=t_amen,   y=r_amen   ]{./box/dat/cme20_conv_symm_rho4.dat};
\addplot+[] table[header=true, x=t_amen_s, y=r_amen_s ]{./box/dat/cme20_conv_symm_rho4.dat};

\addplot+[] table[header=true, x=time, y=resid]{./box/dat/cme20_stop_symm_rho4.dat};

\end{loglogaxis}
\end{tikzpicture}%

%% file: Pic/errors_cme20_symm_i.tikz
\begin{tikzpicture}

\begin{semilogyaxis}[%
cycle list name=amensymm,
xmin=0, xmax=20,
ymin=1e-06, ymax=1,
xlabel={$\log_{10}\mathrm{error}$ vs. iterations},
yminorticks=true]

\addplot+[] table[header=true, x=iter, y=e_dmrg  ]{./box/dat/cme20_conv_symm_rho4.dat};
\addplot+[] table[header=true, x=iter, y=e_dmrg_s]{./box/dat/cme20_conv_symm_rho4.dat};
\addplot+[] table[header=true, x=iter, y=e_alsx  ]{./box/dat/cme20_conv_symm_rho4.dat};
\addplot+[] table[header=true, x=iter, y=e_alsx_s]{./box/dat/cme20_conv_symm_rho4.dat};
\addplot+[] table[header=true, x=iter, y=e_amen  ]{./box/dat/cme20_conv_symm_rho4.dat};
\addplot+[] table[header=true, x=iter, y=e_amen_s]{./box/dat/cme20_conv_symm_rho4.dat};

\addplot+[] table[header=true, x=iter, y=error]{./box/dat/cme20_stop_symm_rho4.dat};

\end{semilogyaxis}
\end{tikzpicture}%

%% file: Pic/errors_cme20_symm_t.tikz
\begin{tikzpicture}

\begin{loglogaxis}[%
cycle list name=amensymm,
xmin=0.1, xmax=10000,
xminorticks=true,
ymin=1e-06, ymax=1,
yminorticks=true,
xlabel={$\log_{10}\mathrm{error}$ vs. $\log_{10}\mathrm{time,~sec}$},
legend style={at={(0.03,0.03)},anchor=south west}]

\addplot+[] table[header=true, x=t_dmrg,   y=e_dmrg  ]{./box/dat/cme20_conv_symm_rho4.dat};
\addplot+[] table[header=true, x=t_dmrg_s, y=e_dmrg_s]{./box/dat/cme20_conv_symm_rho4.dat};
\addplot+[] table[header=true, x=t_alsx,   y=e_alsx  ]{./box/dat/cme20_conv_symm_rho4.dat};
\addplot+[] table[header=true, x=t_alsx_s, y=e_alsx_s]{./box/dat/cme20_conv_symm_rho4.dat};
\addplot+[] table[header=true, x=t_amen,   y=e_amen  ]{./box/dat/cme20_conv_symm_rho4.dat};
\addplot+[] table[header=true, x=t_amen_s, y=e_amen_s]{./box/dat/cme20_conv_symm_rho4.dat};

\addplot+[] table[header=true, x=time, y=error]{./box/dat/cme20_stop_symm_rho4.dat};

\legend{dmrg,dmrg-s, als(t+z),als(t+z)-s, amen,amen-s, stop}
\end{loglogaxis}
\end{tikzpicture}%

%% file: Pic/errors_cme20_amen_i.tikz
\begin{tikzpicture}

\begin{semilogyaxis}[%
cycle list name=varamen,
xmin=0, xmax=20,
ymin=1e-06, ymax=1,
xlabel={$\log_{10}\mathrm{error}$ vs. iterations},
yminorticks=true]

\addplot+[] table[header=true, x=iter, y= e_alsx]{./box/dat/cme20_conv_symm_rho4.dat};
\addplot+[] table[header=true, x=iter, y= e_svd ]{./box/dat/cme20_conv_amen_rho4.dat};
\addplot+[] table[header=true, x=iter, y= e_chol]{./box/dat/cme20_conv_amen_rho4.dat};
\addplot+[] table[header=true, x=iter, y= e_als ]{./box/dat/cme20_conv_amen_rho4.dat};

\addplot+[] table[header=true, x=iter, y=error]{./box/dat/cme20_stop_amen_rho4.dat};

\end{semilogyaxis}
\end{tikzpicture}%

%% file: Pic/errors_cme20_amen_t.tikz
\begin{tikzpicture}

\begin{loglogaxis}[%
cycle list name=varamen,
xmin=0.1, xmax=100,
xminorticks=true,
ymin=1e-06, ymax=1,
yminorticks=true,
xlabel={$\log_{10}\mathrm{error}$ vs. $\log_{10}\mathrm{time,~sec}$},
legend style={at={(0.03,0.03)},anchor=south west}]

\addplot+[] table[header=true, x=t_alsx,y=e_alsx]{./box/dat/cme20_conv_symm_rho4.dat};
\addplot+[] table[header=true, x=t_svd, y=e_svd ]{./box/dat/cme20_conv_amen_rho4.dat};
\addplot+[] table[header=true, x=t_chol,y=e_chol]{./box/dat/cme20_conv_amen_rho4.dat};
\addplot+[] table[header=true, x=t_als, y=e_als ]{./box/dat/cme20_conv_amen_rho4.dat};

\addplot+[] table[header=true, x=time, y=error]{./box/dat/cme20_stop_amen_rho4.dat};

\legend{als(t+z),amen+svd,amen+chol,amen+als,stop}

\end{loglogaxis}
\end{tikzpicture}%

%% file: Pic/errors_cme20_rho.tikz
\begin{tikzpicture}

\begin{semilogyaxis}[%
cycle list name=varamen,
xmin=1, xmax=8,
ymin=1e-06, ymax=1e-04,
xlabel={$\log_{10}\mathrm{error}$ vs. enrichment rank $\rho$},
yminorticks=true]

\pgfplotsset{cycle list shift=1}
\addplot+[] table[header=true, x=rho, y= e_svd]{./box/dat/cme20_rho.dat};
\addplot+[] table[header=true, x=rho, y= e_als]{./box/dat/cme20_rho.dat};

\legend{amen+svd,amen+als};

\end{semilogyaxis}
\end{tikzpicture}%

%% file: Pic/ttimes_cme20_rho.tikz
\begin{tikzpicture}

\begin{axis}[%
cycle list name=varamen,
xmin=1, xmax=8,
ymin=0, ymax=50,
xlabel={time,~sec vs. enrichment rank $\rho$},
yminorticks=true]

\pgfplotsset{cycle list shift=1}
\addplot+[] table[header=true, x=rho, y =t_svd]{./box/dat/cme20_rho.dat};
\addplot+[] table[header=true, x=rho, y =t_als]{./box/dat/cme20_rho.dat};

\end{axis}
\end{tikzpicture}%

%% file: tikz/polymer.tikz
\begin{tikzpicture}[scale=0.3]
\draw (0,0) circle (1);
\node at (3.5,-2.5) {$\mathbf{q}_1$};

\draw[semithick,-] (1,0) to [out=-45,in=135] (1.5,-0.5);
\draw[semithick,-] (1.5,-0.5) to [out=45,in=-135] (2.5,0.5);
\draw[semithick,-] (2.5,0.5) to [out=-45,in=135] (3.5,-0.5);
\draw[semithick,-] (3.5,-0.5) to [out=45,in=-135] (4.5,0.5);
\draw[semithick,-] (4.5,0.5) to [out=-45,in=135] (5.5,-0.5);
\draw[semithick,-] (5.5,-0.5) to [out=45,in=-135] (6,0);

\draw[semithick] (7,0) circle (1);
\node at (8.5,-2.5) {$\mathbf{q}_2$};

\draw[semithick,-] (8,0) to [out=45,in=-135] (8.5,0.5);
\draw[semithick,-] (8.5,0.5) to [out=-45,in=135] (9.5,-0.5);

\node at (12,0) {$\cdots$};

\draw[semithick,-] (13.5,-0.5) to [out=45,in=-135] (14.5,0.5);
\draw[semithick,-] (14.5,0.5) to [out=-45,in=135] (15,0);

\draw[semithick] (16,0) circle (1);
\node at (14.5,-2.5) {$\mathbf{q}_{d-1}$};

\draw[semithick,-] (17,0) to [out=-45,in=135] (17.5,-0.5);
\draw[semithick,-] (17.5,-0.5) to [out=45,in=-135] (18.5,0.5);
\draw[semithick,-] (18.5,0.5) to [out=-45,in=135] (19.5,-0.5);
\draw[semithick,-] (19.5,-0.5) to [out=45,in=-135] (20.5,0.5);
\draw[semithick,-] (20.5,0.5) to [out=-45,in=135] (21.5,-0.5);
\draw[semithick,-] (21.5,-0.5) to [out=45,in=-135] (22,0);

\draw[semithick] (23,0) circle (1);
\node at (20,-2.5) {$\mathbf{q}_d$};
\end{tikzpicture}

%% file: Pic/fp8_tts.tikz

\begin{tikzpicture}

\ifnum\fpplt=1
 \begin{axis}[%
  xmode=normal,ymode=normal,
  color=black,
  mark=*,
  mark size=1pt,
  xmin=0, xmax=20,
  ymin=0, ymax=7,
  yminorticks=true,
  xlabel={$\tau(1,1)$ vs. $t$}
]
\addplot+[] table[header=true, x=t, y=tau11]{./box/dat/fp8_tts.dat};
\fi

\ifnum\fpplt=2
 \begin{axis}[%
  xmode=normal,ymode=normal,
  color=black,
  mark=*,
  mark size=1pt,
  xmin=0, xmax=20,
  ymin=1, ymax=3,
  yminorticks=true,
  xlabel={$\tau(1,2)=\tau(2,1)$ vs. $t$}
]
\addplot+[] table[header=true, x=t, y=tau12]{./box/dat/fp8_tts.dat};
\fi

\ifnum\fpplt=4
 \begin{axis}[%
  xmode=normal,ymode=normal,
  color=black,
  mark=*,
  mark size=1pt,
  xmin=0, xmax=20,
  ymin=-0.02, ymax=0.18,
  yminorticks=true,
  xlabel={$\tau(2,2)$ vs. $t$}
 ]
\addplot+[] table[header=true, x=t, y=tau22]{./box/dat/fp8_tts.dat};
\fi

\end{axis}
\end{tikzpicture}%

%% file: Pic/fp8.tikz

\begin{tikzpicture}

\ifnum\fpplt=0
 \begin{axis}[%
  xmode=normal,ymode=log,
  cycle list name=amen,
  xmin=0, xmax=1,
  ymin=1e2, ymax=1e4,
  yminorticks=true]
\fi

\ifnum\fpplt>0
 \begin{axis}[%
  xmode=normal,ymode=log,
  cycle list name=amen,
  xmin=0, xmax=1,
  ymin=1e-4, ymax=1e-2,
  yminorticks=true,
  legend style={at={(.99,0.01)},anchor=south east}]
\fi

\ifnum\fpplt=0
  \pgfplotsset{xlabel/.add={$\log_{10}\mathrm{time,~sec}$ vs. $\lambda$}}
  \addplot+[] table[header=true, x=lambda, y=t_12]{./box/dat/fp8.dat}; \addlegendentry{$n_r=12$};
  \addplot+[] table[header=true, x=lambda, y=t_16]{./box/dat/fp8.dat}; \addlegendentry{$n_r=16$};
  \addplot+[] table[header=true, x=lambda, y=t_24]{./box/dat/fp8.dat}; \addlegendentry{$n_r=24$};
\fi

\ifnum\fpplt=1
  \pgfplotsset{xlabel/.add={$\|A\psi\|/\|M\psi\|$ vs. $\lambda$}}
  \addplot+[] table[header=true, x=lambda, y=au_12]{./box/dat/fp8.dat};  \addlegendentry{$n_r=12$};
  \addplot+[] table[header=true, x=lambda, y=au_16]{./box/dat/fp8.dat}; \addlegendentry{$n_r=16$};
  \addplot+[] table[header=true, x=lambda, y=au_24]{./box/dat/fp8.dat}; \addlegendentry{$n_r=24$};
\fi

\ifnum\fpplt=2
  \pgfplotsset{xlabel/.add={$\log_{10} |\tau(1,1)-\tau_{ex}(1,1)|$ vs. $\lambda$}}
  \addplot+[] table[header=true, x=lambda, y=tau1_12]{./box/dat/fp8.dat}; \addlegendentry{$n_r=12$};
  \addplot+[] table[header=true, x=lambda, y=tau1_16]{./box/dat/fp8.dat}; \addlegendentry{$n_r=16$};
  \addplot+[] table[header=true, x=lambda, y=tau1_24]{./box/dat/fp8.dat}; \addlegendentry{$n_r=24$};
\fi

\ifnum\fpplt=3
  \pgfplotsset{xlabel/.add={$\log_{10} |\tau(1,2)-\tau_{ex}(1,2)|$ vs. $\lambda$}}
  \addplot+[] table[header=true, x=lambda, y=tau2_12]{./box/dat/fp8.dat}; \addlegendentry{$n_r=12$};
  \addplot+[] table[header=true, x=lambda, y=tau2_16]{./box/dat/fp8.dat}; \addlegendentry{$n_r=16$};
  \addplot+[] table[header=true, x=lambda, y=tau2_24]{./box/dat/fp8.dat}; \addlegendentry{$n_r=24$};
\fi

\ifnum\fpplt=4
  \pgfplotsset{xlabel/.add={$\log_{10} |\tau(2,2)|$ vs. $\lambda$}}
  \addplot+[] table[header=true, x=lambda, y=tau4_12]{./box/dat/fp8.dat}; \addlegendentry{$n_r=12$};
  \addplot+[] table[header=true, x=lambda, y=tau4_16]{./box/dat/fp8.dat}; \addlegendentry{$n_r=16$};
  \addplot+[] table[header=true, x=lambda, y=tau4_24]{./box/dat/fp8.dat}; \addlegendentry{$n_r=24$};
\fi
\end{axis}
\end{tikzpicture}%

%% file: amr2.bbl
\begin{thebibliography}{40}

\bibitem{Ammar-cme-2011}
{\sc A.~Ammar, E.~Cueto, and F.~Chinesta}, {\em Reduction of the chemical
  master equation for gene regulatory networks using proper generalized
  decompositions}, Int. J. Numer. Meth. Biomed. Engng, 00 (2011), pp.~1--15.

\bibitem{Ammar-greedy_fluid_model-2006}
{\sc A.~Ammar, B.~Mokdad, F.~Chinesta, and R.~Keunings}, {\em A new family of
  solvers for some classes of multidimensional partial differential equations
  encountered in kinetic theory modeling of complex fluids}, Journal of
  Non-Newtonian Fluid Mechanics, 139 (2006), pp.~153 -- 176.

\bibitem{griebel-sparsegrids-2004}
{\sc Hans-Joachim Bungatrz and Michael Griebel}, {\em Sparse grids}, Acta
  Numerica, 13 (2004), pp.~147--269.

\bibitem{Lozinski-dilute-2004}
{\sc C\'{e}dric Chauvi\'{e}re and Alexei Lozinski}, {\em {Simulation of dilute
  polymer solutions using a Fokker-Planck equation}}, Computers \& Fluids, 33
  (2004), pp.~687--696.

\bibitem{dc-tt_gmres-2011}
{\sc S.~V. Dolgov}, {\em {TT-GMRES:} on solution to a linear system in the
  structured tensor format}, {arXiv} preprint 1206.5512 (To appear in: {Rus. J.
  of Num. An. and Math. Model.}), 2012.

\bibitem{dkh-cme-2012}
{\sc S.~V. Dolgov and B.~N. Khoromskij}, {\em Tensor-product approach to global
  time-space-parametric discretization of chemical master equation},
  Preprint~68, MPI MIS, 2012.

\bibitem{DKhOs-parabolic1-2012}
{\sc S.~V. Dolgov, Boris~N. Khoromskij, and Ivan~V. Oseledets}, {\em Fast
  solution of multi-dimensional parabolic problems in the tensor
  train/quantized tensor train--format with initial application to the
  {Fokker}-{Planck} equation}, SIAM J. Sci. Comput., 34 (2012),
  p.~A3016–A3038.

\bibitem{DoOs-dmrg-solve-2011}
{\sc S.~V. Dolgov and I.~V. Oseledets}, {\em Solution of linear systems and
  matrix inversion in the {TT}-format}, SIAM J. Sci. Comput., 34 (2012),
  pp.~A2718--A2739.

\bibitem{ds-amr1-2013}
{\sc S.~V. Dolgov and D.~V. Savostyanov}, {\em Alternating minimal energy
  methods for linear systems in higher dimensions. {Part I}: {SPD} systems},
  {arXiv} preprint 1301.6068, 2013.

\bibitem{fannes-mps-1992}
{\sc M.~Fannes, B.~Nachtergaele, and R.F. Werner}, {\em Finitely correlated
  states on quantum spin chains}, Communications in Mathematical Physics, 144
  (1992), pp.~443--490.

\bibitem{gillespie-ssa-1976}
{\sc D.T. Gillespie}, {\em A general method for numerically simulating the
  stochastic time evolution of coupled chemical reactions}, Journal of
  computational physics, 22 (1976), pp.~403--434.

\bibitem{hackbusch-2012}
{\sc W.~Hackbusch}, {\em Tensor spaces and numerical tensor calculus},
  Springer--Verlag, Berlin, 2012.

\bibitem{hegland-cme-2007}
{\sc Markus Hegland, Conrad Burden, Lucia Santoso, Shev MacNamara, and Hilary
  Booth}, {\em A solver for the stochastic master equation applied to gene
  regulatory networks}, Journal of Computational and Applied Mathematics, 205
  (2007), pp.~708 -- 724.

\bibitem{holtz-ALS-DMRG-2012}
{\sc S.~Holtz, T.~Rohwedder, and R.~Schneider}, {\em The alternating linear
  scheme for tensor optimization in the tensor train format}, SIAM J. Sci.
  Comput., 34 (2012), pp.~A683--A713.

\bibitem{jahnke-cme-2008}
{\sc Tobias Jahnke and Wilhelm Huisinga}, {\em A dynamical low-rank approach to
  the chemical master equation}, Bulletin of Mathematical Biology, 70 (2008),
  pp.~2283--2302.

\bibitem{kkns-cme-2013}
{\sc V.~Kazeev, M.~Khammash, M.~Nip, and C.~Schwab}, {\em Direct solution of
  the chemical master equation using quantized tensor trains}, Research
  Report~04, SAM, ETH Z\"{u}rich, 2013.

\bibitem{khor-qtt-2011}
{\sc B.~N. Khoromskij}, {\em $\mathcal{O}(d \log n)$--{Quantics} approximation
  of {$N$--$d$} tensors in high-dimensional numerical modeling}, Constr. Appr.,
  34 (2011), pp.~257--280.

\bibitem{khor-survey-2011}
\leavevmode\vrule height 2pt depth -1.6pt width 23pt, {\em Tensor-structured
  numerical methods in scientific computing: {Survey} on recent advances},
  Chemometr. Intell. Lab. Syst., 110 (2012), pp.~1--19.

\bibitem{klumper-mps-1993}
{\sc A.~Kl\"{u}mper, A.~Schadschneider, and J.~Zittartz}, {\em Matrix product
  ground states for one-dimensional spin-1 quantum antiferromagnets}, Europhys.
  Lett., 24 (1993), pp.~293--297.

\bibitem{kolda-review-2009}
{\sc T.~G. Kolda and B.~W. Bader}, {\em Tensor decompositions and
  applications}, SIAM Review, 51 (2009), pp.~455--500.

\bibitem{Lozinski-2D_FENE-2003}
{\sc A.~Lozinski and C.~Chauvi\'{e}re}, {\em {A fast solver for Fokker-Planck
  equation applied to viscoelastic flows calculations: 2D FENE model}}, Journal
  of Computational Physics, 189 (2003), pp.~607 -- 625.

\bibitem{nouy-modred-2010}
{\sc A.~Nouy}, {\em A priori model reduction through proper generalized
  decomposition for solving time-dependent partial differential equations},
  Computer Methods in Applied Mechanics and Engineering, 199 (2010),
  pp.~1603--1626.

\bibitem{Os-mvk2-2011}
{\sc I.~V. Oseledets}, {\em {DMRG} approach to fast linear algebra in the
  {TT}--format}, Comput. Meth. Appl. Math, 11 (2011), pp.~382--393.

\bibitem{osel-tt-2011}
\leavevmode\vrule height 2pt depth -1.6pt width 23pt, {\em Tensor-train
  decomposition}, SIAM J. Sci. Comput., 33 (2011), pp.~2295--2317.

\bibitem{ot-ttcross-2010}
{\sc I.~V. Oseledets and E.~E. Tyrtyshnikov}, {\em {TT-cross} approximation for
  multidimensional arrays}, Linear Algebra Appl., 432 (2010), pp.~70--88.

\bibitem{ostlund-dmrg-1995}
{\sc S.~{\"Ostlund} and S.~Rommer}, {\em Thermodynamic limit of density matrix
  renormalization}, Phys. Rev. Lett., 75 (1995), pp.~3537--3540.

\bibitem{risken-fpe-1989}
{\sc Hannes Risken}, {\em The Fokker-Planck Equation: Methods of Solutions and
  Applications, 2nd ed.}, Springer Verlag, Berlin, Heidelberg, 1989.

\bibitem{ushmaev-tt-2013}
{\sc T.~Rohwedder and A.~Uschmajew}, {\em Local convergence of alternating
  schemes for optimization of convex problems in the {TT} format}, SIAM J Num.
  Anal.,  ((2013)).
\newblock to appear.

\bibitem{saad-iter}
{\sc Y.~Saad}, {\em Iterative methods for sparse linear systems}, SIAM, 2003.

\bibitem{sav-rr-2009}
{\sc D.~V. Savostyanov}, {\em Fast revealing of mode ranks of tensor in
  canonical format}, Numer. Math. Theor. Meth. Appl., 2 (2009), pp.~439--444.

\bibitem{so-dmrgi-2011proc}
{\sc D.~V. Savostyanov and I.~V. Oseledets}, {\em Fast adaptive interpolation
  of multi-dimensional arrays in tensor train format}, in Proceedings of 7th
  International Workshop on Multidimensional Systems (nDS), IEEE, 2011.

\bibitem{smolyak-1963}
{\sc S.~A. Smolyak}, {\em Quadrature and interpolation formulas for tensor
  products of certain class of functions}, Dokl. Akad. Nauk SSSR, 148 (1964),
  pp.~1042--1053.
\newblock Transl.: Soviet Math. Dokl. 4:240-243, 1963.

\bibitem{trefethen-spectral-2000}
{\sc Lloyd~N. Trefethen}, {\em Spectral methods in MATLAB}, SIAM, Philadelphia,
  2000.

\bibitem{vankampen-stochastic-1981}
{\sc N.~G. van Kampen}, {\em Stochastic processes in physics and chemistry},
  North Holland, Amsterdam, 1981.

\bibitem{Venkiteswaran-QMC_FP-2005}
{\sc G.~Venkiteswaran and M.~Junk}, {\em {A QMC approach for high dimensional
  Fokker-Planck equations modelling polymeric liquids}}, Math. Comput. Simul.,
  68 (2005), pp.~43--56.

\bibitem{white-dmrg-1993}
{\sc S.~R. White}, {\em Density-matrix algorithms for quantum renormalization
  groups}, Phys. Rev. B, 48 (1993), pp.~10345--10356.

\end{thebibliography}
